\documentclass[sn-mathphys,Numbered]{sn-jnl}

\usepackage{graphicx}%
\usepackage{multirow}%
\usepackage{amsmath,amssymb,amsfonts}%
\usepackage{amsthm}%
\usepackage{mathrsfs}%
\usepackage[title]{appendix}%
\usepackage{xcolor}%
\usepackage{textcomp}%
\usepackage{manyfoot}%
\usepackage{booktabs}%
\usepackage{algorithm}%
\usepackage{algorithmicx}%
\usepackage{algpseudocode}%
\usepackage{listings}%
\usepackage{mathtools}
\usepackage{subcaption}
\usepackage{todonotes}

\newcommand{\R}{\mathbb{R}}

\renewcommand{\vec}[1]{\mathbf{#1}}

\newcommand{\half}{{\textstyle \frac{1}{2}}}

\newcommand{\IR}{\mathbb{R}}

\newcommand{\norm}[1]{\left\lVert #1\right\rVert}

\newcommand{\ou}{\overline{u}}
\newcommand{\0}{\mathbf{0}}
\newcommand{\bb}{\mathbf{b}}
\newcommand{\nn}{\mathbf{n}}
\newcommand{\uu}{\mathbf{u}}

\newcommand{\xx}{\mathbf{x}}
\newcommand{\yy}{\mathbf{y}}

\newif\ifshowrevisions
\showrevisionstrue
\DeclareRobustCommand{\rev}[1]{{\ifshowrevisions\color{black}\fi#1}}
\newenvironment{revblock}{\ifshowrevisions\color{black}\fi}{}

\theoremstyle{thmstyleone}%
\newtheorem{theorem}{Theorem}
\newtheorem{proposition}[theorem]{Proposition}%
\newtheorem{lemma}{Lemma}%

\theoremstyle{thmstyletwo}%
\newtheorem{remark}{Remark}%

\theoremstyle{thmstylethree}%

\newcommand{\junk}[1]{{}}

\begin{document}

\title[%
Finite Volume Convergence for a Heat Equation with a Measure-Valued Source
]{%
Finite Volume Convergence for a Heat Equation with a Measure-Valued Source
}


\author[1]{\fnm{Matthias K.} \sur{Gobbert}}\email{gobbert@umbc.edu}

\author[2]{\fnm{Martin} \sur{Kru\v{z}\'{\i}k}}\email{kruzik@utia.cas.cz}

\author[3]{\fnm{Andreas} \sur{Meister}}\email{meister@mathematik.uni-kassel.de}

\author[1]{\fnm{Michael} \sur{Muscedere}}\email{mmusce1@umbc.edu}



\affil[1]{\orgdiv{Department of Mathematics and Statistics}, \orgname{University of Maryland, Baltimore County}, \orgaddress{\street{1000 Hilltop Circle}, \postcode{MD 21250} \city{Baltimore}, \country{USA}}}

\affil[2]{\orgname{Czech Academy of Sciences, Institute of Information Theory and Automation}, \orgaddress{\street{Pod vod\'{a}renskou v\v{e}\v{z}\'{\i} 4}, \postcode{182 00} \city{Praha 8}, \country{Czechia}} \&
\orgname{Czech Technical University, Faculty of Civil Engineering} \orgaddress{\street{Th\'{a}kurova 7}, \postcode{166 29} \city{Praha 6}, \country{Czechia}}}

\affil[3]{\orgdiv{Institute of Mathematics}, \orgname{University of Kassel}, \orgaddress{\street{Untere K\"onigsstr.~86}, \postcode{34117} \city{Kassel}, \country{Germany}}}




\abstract{%
Past work with Tom Seidman analyzed the spatial convergence order
of the numerical finite element method applied to
time-dependent diffusion-reaction equations with a measure-valued source term.
For smooth source terms, the numerical method has second-order accuracy,
but the non-smooth source terms introduce
a limitation of the order that depends on the spatial dimension.
This work in honor of Tom extends the analysis to
the finite volume method \rev{in its classical two-point flux form} and shows that the regularity of the source term
reduces the convergence order to 1 in two and 0.5 in three spatial dimensions
which is reduced from the nominal second-order and depends on the space dimension.
Computational tests with a point source modeled by a Dirac delta distribution
exhibit the behavior predicted by the analytical results.
}

\keywords{%
Diffusion-reaction equation;
Dirac delta distribution;
Finite element method;
Finite volume method;
Spatial convergence order}



\maketitle
\begin{center}
    {\it This paper is dedicated to the memory of Tom Seidman.}
\end{center}


\section{Introduction}

For co-author MG, Thomas~I. ``Tom'' Seidman was a dear departmental colleague, and many memories
connect to him, from the job interview, many conversations and visits to each
other's offices which led to joint papers on various topics,
e.g., \cite{SoaneGobbertSeidmanNW2004,GobbertEtAlCOMSOL2009},
co-teaching a special topics course that leveraged complementary
skills on the subject,
to serving as the lead-organizer for his 70th birthday conference in 2007
for his 50~years of research \cite{GobbertHoffmanShenIEEE2007}.
Tom also served on thesis committees, where he became aware of the question
of convergence order of the finite element method in the face of
thousands of point sources modeling calcium injection sites in a heart cell
as Dirac delta distributions \cite{Hanhartthesis2002,Hanhartjcam2002}.
Tom was never afraid to tackle anything, and
in this spirit he also engaged with numerical topics without hesitation,
thus starting the work on this project,
based on the purely numerical evidence in
\cite{Hanhartjcam2002,GobbertSISC2007,TrottGobbertCOMSOL2011}.
Co-author MM in the same way was connected to Tom by many conversations,
colloquium talks, the conference, and a joint research project with MG and Tom.
The other co-authors MK and AM enjoyed the wide-ranging conversations with Tom
during their visits to the department at UMBC, where the same question was
discussed, leading with co-author MK to the 
joint analysis paper of the finite element method for this problem
\cite{Measure_NumerMath2011} and
with co-author AM to the implementation of the
finite volume method for this application problem \cite{CalciumFV_NMPDE2012}.
The occasion of this special issue in honor of Tom Seidman has now brought
the co-authors together for this paper that is
an extension of \cite{Measure_NumerMath2011}
from the finite element method (FEM) to 
the finite volume method (FVM).

\junk{
papers on a range of topics; MG collaboration in department
\cite{SoaneGobbertSeidmanNW2004,GobbertEtAlCOMSOL2009,Measure_NumerMath2011}
thesis committee \cite{Hanhartthesis2002},
co-taught a course, led organization of 70th birthday conference in 2007 for 50 years of research career \cite{GobbertHoffmanShenIEEE2007};
he was never afraid to tackle anything;
in this spirit he also engaged with numerical topics without hesitation;
thus developed the work on this project,
originally with MK upon his visit to UMBC when noticing numerical evidence
in \cite{Hanhartthesis2002,Hanhartjcam2002};
AM also visited and was inspired by conversations with Tom led to \cite{CalciumFV_NMPDE2012},
the complementary computational paper to \cite{Measure_NumerMath2011};
We are thus happy to report that on the occasion of this special issue in honor of Tom Seidman,
the authors got together for this paper.
It is a generalization of \cite{Measure_NumerMath2011} from FEM
to other numerical methods for the same problem, notably FVM.
}

Like in \cite{Measure_NumerMath2011} for finite elements,
we are interested in the convergence behavior of finite volume schemes
applied to diffusion-reaction equations of the form
\begin{align}
  u_t(\xx,t) - \nabla \cdot \left(D \, \nabla u(\xx,t)\right) &= \varphi(\xx, t)
    &&\textrm{in} \; {\cal{Q}} = \Omega \times \R_0^+, \label{eq:DREQ} \\
  u(\xx,t) &= 0
    &&\textrm{on} \; \partial\Omega \times \R_0^+ \label{DBC} \\
  u(\xx,0) &= 0
    &&\textrm{on} \; \overline{\Omega}, \label{eq:Uinit}
\end{align}
where  $\Omega\subset\R^d$ is a 
bounded Lipschitz domain in $d$ spatial dimensions,
$\varphi$ includes a smooth or measure-valued source term,
combined with  Dirichlet boundary conditions, and an initial condition.

In \cite[Theorem~4.2]{Measure_NumerMath2011}, the finite element solution $u_h$ on a mesh with spacing $h$
was shown to have a dimension-dependent convergence order in the $L^2$-norm with order
$\|u(t)-u_h(t)\|_{L^2(\Omega)} \leq C h^\lambda$ with
$\lambda \approx {2 - d/2}$ in $d = 2$ and $3$ dimensions,
that is, order $h$ in 2-D and $h^{1/2}$ in 3-D.
We are now able to extend the result to the finite volume method\rev{,
in the form of the classical two-point flux approximation described in Section~\ref{sec:num_method}}, with result
\begin{align}
\|u(t)-u_h(t)\|_{L^2(\Omega)} &\le 
\begin{cases}
C h |\log h|^{1/2}, & d=2,\\
C h^{1/2}, & d=3,
\end{cases}
\end{align}
in (\ref{equ:ResultHeatPersistentSemi}) in Section~\ref{sec:HeatErrAnalPersistent} below,
that shows essentially the same convergence behavior.
The question of extending the finite element results to the finite volume method
had come up among the co-authors ever since the numerical results
in \cite{CalciumFV_NMPDE2012}, of which Tom was well-aware.
Unfortunately, we could not find the time to extend the results together with Tom,
but it is very fitting that we have been able to come together on the occasion of this paper.

For the model heat equation \eqref{eq:DREQ} with $D = 1$ and $\varphi = \mu$,
where \(\mu\) is a bounded Radon measure, the problem is well understood at the
level of the continuous PDE. In particular, the theory of renormalized solutions
for nonlinear parabolic equations with general measure data includes the heat
equation as the special case \(p=2\) of
\[
    u_t - \Delta_p u = \mu. 
\]
A relevant reference is Petitta~\cite{Petitta2008}.

For finite volume methods, the situation appears to be less complete. There are
finite volume convergence results for parabolic convection--diffusion equations
with \(L^1\) right-hand sides, for instance
Gallou{\"e}t, Larcher and Latch{\'e}~\cite{GallouetLarcherLatche2012}.
There are also finite volume results for nonlinear parabolic problems with
\(L^1\) data in the framework of renormalized solutions, see
Guib{\'e} and Leclavier~\cite{GuibeLeclavier2019}. However, a direct finite
volume analogue of the elliptic result of Droniou, Gallou{\"e}t and
Herbin~\cite{DroniouGallouetHerbin2003}, with a genuinely Radon-measure
right-hand side, does not seem to be as standard for the heat equation.

Thus, one may summarize the state of the literature as follows. The heat
equation with measure-valued right-hand side is solved at the PDE level. For
finite volume schemes, the available convergence theory is well developed for
\(L^1\) parabolic data, while the fully Radon-measure finite volume theory is
clearly established in the elliptic case and seems less explicit in the standard parabolic finite volume literature. From the point of view of applications, 
heat equations with low-regularity/measure right-hand side appear in many models of the continuum mechanics of solids, as e.g. in modeling of shape memory alloys \cite{kruzik-roubicek}. 

\rev{Since the analysis concerns a specific finite volume scheme,
Section~\ref{sec:num_method} first describes this scheme,
namely the classical two-point flux approximation, which is the simplest
finite volume method for diffusion problems,
and introduces the notation used in the analysis.
Section~\ref{sec:num_method} also highlights key features of
the finite volume method and explains why it is particularly
suitable for problems with thousands of
measure-valued Dirac source terms that motivated this work originally
\cite{Hanhartthesis2002,Hanhartjcam2002}.}
The following analysis uses the same approach developed with Tom
in \cite{Measure_NumerMath2011} by first
considering the stationary problem corresponding to \eqref{eq:DREQ}
before returning to the time-dependent problem.
More precisely, \rev{Section~\ref{sec:PoisErrAnal}}
considers the Poisson equation with a Dirac source, before
Sections~\ref{sec:HeatErrAnalSpaceTime} and~\ref{sec:HeatErrAnalPersistent}
consider the heat equation with a space-time Dirac and a persistent Dirac source, respectively.
Each section analyzes the convergence in $d = 2$ and $d = 3$ space dimensions
and demonstrates the effect of dimensionality in the final results
\eqref{equ:ResultPois2D} and \eqref{equ:ResultPois3D} in Section~\ref{sec:PoisErrAnal},
\eqref{equ:ResultHeatSpaceTime2D} and \eqref{equ:ResultHeatSpaceTime3D}
in Section~\ref{sec:HeatErrAnalSpaceTime}, and 
\eqref{equ:ResultHeatPersistentSemi} and
\eqref{equ:ResultHeatPersistent} in Section~\ref{sec:HeatErrAnalPersistent}.
\rev{For the Poisson equation, the scaling arguments of
Subsections~\ref{subsec:Pois2D} and~\ref{subsec:Pois3D} are complemented in
Subsection~\ref{subsec:PoisRigorous} by a proof that also accounts for the global
propagation of local truncation errors by the discrete solution operator.}

\rev{Finally, Section~\ref{sec:num_results} provides} numerical demonstrations that
the behavior predicted by the analysis is exhibited in actual computations\rev{;
it} is comprised of two subsections
for the stationary and time-dependent problems, respectively.
Specifically, the convergence orders
in Table~\ref{tab:PoissonDiracDirichlet} in Subsection~\ref{subsec:stationary}
exhibit the behavior predicted in
\eqref{equ:ResultPois2D} and \eqref{equ:ResultPois3D} in Section~\ref{sec:PoisErrAnal}
and the convergence orders
in Table~\ref{tab:HeatDiracDirichlet} in Subsection~\ref{subsec:transient}
the behavior predicted in
\eqref{equ:ResultHeatPersistent} in Section~\ref{sec:HeatErrAnalPersistent}.
Both tables' computational results point to the hypothesis of
the convergence order being $2 - d/2 = 1$ and $1/2$
in $d = 2$ and $3$ space dimensions, respectively.

In summary, this paper shows the analytical proof for the
extension of the old result for the finite element method,
obtained with Tom Seidman in \cite{Measure_NumerMath2011},
to the finite volume method,
and potentially to other methods like the discontinuous Galerkin method.
The accompanying computational evidence confirms the analytical
behavior and is obtained by sample code
in Matlab for test problems in 2 and 3 spatial dimensions
that aligns with previous results from general purpose C code.



\section{\rev{Finite Volume Method Discretization}}\label{sec:num_method}

\subsection{General formulation}

The essence of the finite volume method, see e.g. \cite{barrenechea2025} lies in integrating the PDE \eqref{eq:DREQ} of the problem
over so-called control volumes $\sigma_k \subset \overline{\Omega}$ to derive an integral form
of the PDE.
To this end, the open domain $\Omega$ is decomposed into a
finite number of disjoint open control volumes $\sigma_1, \ldots, \sigma_m \subset \Omega$
such that $\cup_{k=1}^m \overline{\sigma_k} = \overline{\Omega}$.
In our context of a regular quadrilateral domain $\Omega = (-50,50)^d$
in $d = 1, 2, 3$ dimensions in Section~\ref{sec:num_results},
we use a regular mesh with spacing $h$ and $N$ control volumes in each dimension,
for a total of $m = N^d$ volumes.
The integration of the PDE \eqref{eq:DREQ} over each control volume yields then
\begin{equation}\label{eq:INT}
   \frac{d}{dt} \int_{\sigma_k} u(\xx,t) \, d\xx
     = \int_{\sigma_k} \nabla \cdot \left(D \, \nabla u(\xx,t)\right) \, d\xx
     + \int_{\sigma_k} \varphi(\xx,t) \, d\xx,
\end{equation}
where time differentiation and spatial integration were already interchanged.
Introducing the cell averages
\begin{equation}\label{eq:CA}
  u_k(t) = \frac{1}{|\sigma_k|} \int_{\sigma_k} u(\xx,t) \, d\xx
\end{equation}
makes the following integral form of \eqref{eq:INT}
\begin{equation}\label{eq:INT1}
	\frac{d}{dt} u_k(t)
    = \frac{1}{|\sigma_k|} \int_{\partial \sigma_k} \nn \cdot \left(D \, \nabla u(\xx,t)\right) \, ds
    + \frac{1}{|\sigma_k|} \int_{\sigma_k} \varphi(\xx,t) \, d\xx
\end{equation}
eventually a system of equations for the cell averages as unknowns,
where the integral over the diffusion term is converted to a surface integral
using the Gaussian integral theorem with outer unit normal vector $\nn$.

\begin{revblock}
It remains to approximate the normal fluxes in \eqref{eq:INT1}.
Since \eqref{eq:DREQ} contains no convective term, the only fluxes are the diffusive fluxes $\nn\cdot(D\nabla u)$,
which depend linearly on $u$ and are continuous across the faces of the control volumes.
Consequently, no reconstruction of left and right states at the faces and no (approximate) Riemann solver are needed,
in contrast to finite volume schemes for hyperbolic or convection-dominated problems.
We use the classical two-point flux approximation \cite{EymardGallouetHerbin2000},
which is the simplest finite volume method for diffusion problems:
each control volume $\sigma_k$ carries one unknown $u_k$, which approximates the cell average \eqref{eq:CA}
and is associated with a point $\xx_k\in\overline{\sigma_k}$,
and the flux through an interior face $e=\partial{\sigma_k}\cap\partial{\sigma_l}$ is approximated by the difference quotient
\begin{equation}\label{eq:TPFA}
\int_e \nn_{k,e}\cdot\left(D\,\nabla u\right) ds \approx F_{k,e}(\uu) := D\,|e|\,\frac{u_l-u_k}{d_e},
\end{equation}
where $\nn_{k,e}$ is the unit normal pointing out of $\sigma_k$, $d_e=|\xx_k-\xx_l|$,
and the segment from $\xx_k$ to $\xx_l$ is orthogonal to $e$ (admissibility of the mesh).
For a face $e\subset\partial\Omega$, $u_l$ is replaced by the boundary value $u=0$
and $d_e$ by the distance from $\xx_k$ to $e$.
The fluxes are conservative, $F_{l,e}=-F_{k,e}$.
Each unknown thus interacts only with its face neighbors, and no higher-order reconstruction is used.
This is what we mean by a \emph{low-order}, more precisely, lowest-order, finite volume method throughout this paper;
its accuracy is summarized in Subsection~\ref{subsec:FVnotation} below.
On the uniform Cartesian meshes of Section~\ref{sec:num_results}, \eqref{eq:TPFA} reduces in the interior of the domain
to central difference quotients, i.e., to the classical five-point ($d=2$) and seven-point ($d=3$) stencils,
which leads to the Kronecker product form of the matrices given below.
\end{revblock}

The introduction of the surface integral of the diffusion term
in the first term on the right-hand side of \eqref{eq:INT1} allows the
finite volume method to conserve the mass in the system at the discrete level,
which is one key advantage in application problems such as the
calcium flow in a heart cell \cite{Hanhartthesis2002,Hanhartjcam2002}
that motivated this work originally.
The other key advantage lies in the ability of the integral in the
second term on the right-hand side of \eqref{eq:INT1} to rigorously handle
a point source modeled by a Dirac delta $\varphi(\xx,t) = \kappa \, \delta_0(\xx)$
with injection rate $\kappa$ as
\begin{equation}
    \frac{1}{|\sigma_k|} \int_{\sigma_k} \varphi(\xx,t) \, d\xx
    = \frac{1}{|\sigma_k|} \int_{\sigma_k} \kappa \, \delta_0(\xx) \, d\xx
    = \frac{\kappa}{|\sigma_k|}
\end{equation}
for that control volume $\sigma_k$ that contains the location of the injection site
and $0$ otherwise.
An \rev{analogue} of this is possible for the finite element method as in \cite{Measure_NumerMath2011},
but it is not possible for instance for a finite difference method that cannot resolve
the Dirac delta.
\rev{In particular, the finite volume method uses only the integrals of the source over the control volumes.
It therefore does not distinguish between the Dirac source $\kappa\,\delta_0$ and the regularized source
$\kappa\,|\sigma_k|^{-1}\mathbf{1}_{\sigma_k}$ that is uniformly distributed over the control volume containing the injection site;
this observation is used in the analysis in Section~\ref{sec:HeatErrAnalPersistent}.}

For the vector of unknown cell averages as function of time
and with the vector of integrals of the source term
\[
\uu(t) = \begin{pmatrix}
	               u_1(t)\\
	               \vdots \\
	               u_{m}(t)
	            \end{pmatrix},
\quad
\mathbf{\Phi}(t)
\begin{pmatrix}
	\frac{1}{|\sigma_1|} \int_{\sigma_1} \varphi(\xx,t) \, d\xx\\
	\vdots \\
	\frac{1}{|\sigma_m|} \int_{\sigma_{m}} \varphi(\xx,t) \, d\xx	
\end{pmatrix},
\] 
the semi-discrete form of the finite volume method can then be formulated in vector notation as
\begin{equation}\label{eq:semi}
	\hat{M} \frac{d\uu(t)}{dt} 
	= A\,\uu(t) + \hat{M} \mathbf{\Phi}(t)
\end{equation}
with lumped mass matrix $\hat{M} \in \R^{m \times m}$
and stiffness matrix $A \in \R^{m \times m}$ with $m = N^d$ in $d = 1, 2, 3$ spatial dimensions
for $N$ the number of control volumes in each spatial dimension.
These matrices are block diagonal or \rev{tridiagonal} matrices with diagonal or \rev{tridiagonal} blocks, respectively,
that can be very effectively pre-computed as sums of Kronecker products $\otimes$ in software such as Matlab.
For instance,
the matrices are $\hat{M} = S \in \IR^{N \times N}$ and
$A = \frac{D}{h^2} \, T\in \R^{N \times N}$ in the one-dimensional case,
$\hat{M} = S \otimes S \in \R^{N^2 \times N^2}$ and
$A = \frac{D}{h^2}(S \otimes T + T \otimes S) \in \R^{N^2 \times N^2}$ in two dimensions, and
$\hat{M} = S \otimes S \otimes S \in \R^{N^3 \times N^3}$ and
$A = \frac{D}{h^2}(S \otimes S \otimes T + S \otimes T \otimes S + T \otimes S \otimes S) \in \R^{N^3 \times N^3}$
in three dimensions, constructed from
the diagonal matrix $S \in \IR^{N \times N}$ and
the tridiagonal matrix $T \in \IR^{N \times N}$.
Matlab's sparse storage scheme is the key to the efficiency of
operations such as matrix-vector products with these matrices.

The exact coefficients in the matrices $S$, $T$, $A$, and $\hat{M}$ depend on the
discretization of the diffusion term and the boundary conditions and are designed
to facilitate the Kronecker products for efficient implementation,
resulting in
$S \in \R^{N \times N}$ diagonal with $\half, 1,\ldots, 1, \half$ on the diagonal and
$T \in \R^{N \times N}$ tridiagonal with $1, 2, \ldots, 2, 1$ on the main diagonal
and $-1$ on the adjacent lower and upper bands.
In the current implementation, we use central differences in the interior of the domain and
one-sided approximations along the boundary\rev{; in the interior, this is exactly the two-point flux
approximation \eqref{eq:TPFA} on the uniform mesh}.
The original method uses flux boundary conditions yielding the specific matrices above.
The Dirichlet boundary conditions are implemented as correction of certain values in $A$,
namely all diagonal values, including the ones corresponding to control volumes at
the boundary, being the same value as the diagonal value for interior control volumes.

The time discretization of the semi-discrete problem \eqref{eq:semi} results necessarily
in a stiff system of ordinary differential equations, which necessitates the use of an appropriate ODE solver.
The implicit Euler time-stepping method is both appropriate for this situation, simple, and sufficiently accurate if using small time steps.
Solving the implicit system for the solution at every new time step by the
iterative conjugate gradient method
is highly accurate if a tight tolerance is used and still very efficient, since the solution at the old time step can be used as an excellent initial guess.

For a stationary test problem, the left-hand side of \eqref{eq:semi} is not present
and no variable depends on time,
thus we have one linear system for the unknowns to solve, which we solve by \rev{sparse} direct
Gaussian elimination\rev{, since it is robust and requires no tunable parameters such as stopping tolerances.
For the symmetric, diagonally dominant system matrix $A$, Gaussian elimination is backward stable,
so that the relative error of the computed solution is bounded by approximately $\kappa(A)\,\varepsilon_{\mathrm{mach}}$,
where $\kappa(A)$ denotes the condition number of $A$ and $\varepsilon_{\mathrm{mach}}\approx1.1\cdot10^{-16}$
the unit roundoff of double-precision arithmetic \cite{Higham2002}.
The computed solution is therefore not accurate to full double precision:
here, $\kappa(A)=O(h^{-2})$ grows under mesh refinement, which is also the origin of the stiffness
of the time-dependent problem noted above.
However, even on the finest meshes used in Section~\ref{sec:num_results}, the resulting algebraic error is
many orders of magnitude smaller than the discretization errors reported in Table~\ref{tab:PoissonDiracDirichlet}
and therefore does not affect the convergence study.}

\junk{%
Using the implicit Euler method with a time step size $\Delta t$ and $t^0 = 0$ as well as $t^{n+1} = t^n + \Delta t$, $n=0,1,\ldots,$ we obtain the fully discrete method in the form $M\uu^{n+1} = F(t^n,t^{n+1})$
\[
 \left(I - \frac{\Delta t}{h^2} A\right)\uu^{n+1} = \uu^n + \frac{\Delta t}{h^2} \left( \bb^{n+1} + \mathbf{\Phi}^{n+1}\right),
\]
where $\uu^n$ and $\uu^{n+1}$ represent approximations of the corresponding vector of cell averages $\uu(t^n)$ and $\uu(t^{n+1})$ respectively, while $\bb^{n+1} = \bb(t^{n+1})$ as well as $\mathbf{\Phi}^{n+1} = \mathbf{\Phi}(t^{n+1})$.
}

The size of the system in \eqref{eq:semi} is given by the number $m$ of control volumes.
This becomes a significant scale in higher dimensions as $m = N^d$ in $d = 1, 2, 3$ spatial dimensions
for $N$ the number of control volumes in each spatial dimension.
The method is implemented in Matlab, fully taking advantage of its efficient vectorization capabilities.

\begin{revblock}
\subsection{Notation and accuracy of the scheme used in the analysis}\label{subsec:FVnotation}

In the analysis of Sections~\ref{sec:PoisErrAnal}--\ref{sec:HeatErrAnalPersistent}, we set $D=1$,
denote the mesh by $\mathcal{T}_h=\{\sigma_1,\ldots,\sigma_m\}$, write $K,L\in\mathcal{T}_h$ for generic control volumes,
$\mathcal{E}$ for the set of faces, and $h$ for the mesh size.
The mesh is assumed to be quasi-uniform, i.e., $\operatorname{diam}(K)\sim h$ and $|K|\sim h^d$,
and $K_0$ denotes the control volume containing the source point $\xx_0$.
A vector of unknowns $v=(v_K)_{K\in\mathcal{T}_h}$ is identified with the piecewise constant function that equals $v_K$ on $K$,
and $P_h$ denotes the projection onto cell averages, $(P_hu)_K=|K|^{-1}\int_K u\,d\xx$.
We use the discrete inner product, norm, and energy seminorm
\begin{equation}\label{eq:discrete-norms}
(v,w)_h=\sum_{K\in\mathcal{T}_h}|K|\,v_Kw_K,\qquad
\|v\|_h=(v,v)_h^{1/2},\qquad
|v|_{1,h}^2=\sum_{e\in\mathcal{E}}\frac{|e|}{d_e}\,(D_ev)^2,
\end{equation}
where $D_ev=v_K-v_L$ for an interior face $e$ between $K$ and $L$ (with a fixed orientation),
$D_ev=v_K$ for a face $e\subset\partial K\cap\partial\Omega$, and $d_e$ is the corresponding distance in \eqref{eq:TPFA}.
Note that $\|v\|_h$ is the $L^2(\Omega)$ norm of the piecewise constant function $v$,
and that $|\cdot|_{1,h}$ is the discrete energy norm (discrete $H^1_0$ seminorm),
which is denoted by $\|\cdot\|_{\mathrm{energy}}$ in Section~\ref{sec:HeatErrAnalSpaceTime}.
The discrete Laplacian is defined by $(\Delta_hv)_K=|K|^{-1}\sum_{e\subset\partial K}F_{K,e}(v)$
with the fluxes \eqref{eq:TPFA} for $D=1$, so that
\[
(-\Delta_hv,w)_h=\sum_{e\in\mathcal{E}}\frac{|e|}{d_e}\,D_ev\,D_ew=:a_h(v,w),\qquad |v|_{1,h}^2=a_h(v,v).
\]
The finite volume approximations of $-\Delta u=f$ and of $u_t-\Delta u=f$ then read
\[
-\Delta_hu_h=f_h
\qquad\text{and}\qquad
u_h'-\Delta_hu_h=f_h,
\qquad\text{with } (f_h)_K=|K|^{-1}\int_Kf\,d\xx .
\]
In particular, $f=\delta_{\xx_0}$ leads to $f_h=\mu_h:=|K_0|^{-1}\mathbf{1}_{K_0}$,
which is also the source vector of the regularized source $\mu_h$; this is used in Section~\ref{sec:HeatErrAnalPersistent}.
The matrix of $a_h$ is a symmetric M-matrix, so that the scheme satisfies a discrete maximum principle.

Throughout this paper, \emph{low-order} refers to this lowest-order finite volume method
with one unknown per control volume, piecewise constant approximations, and two-point fluxes without reconstruction.
For smooth solutions on the uniform meshes used in Section~\ref{sec:num_results},
it is second-order accurate in the discrete $L^2$ norm, $\|P_hu-u_h\|_h=O(h^2)$,
and (at least) first-order accurate in the discrete energy norm, $|P_hu-u_h|_{1,h}=O(h)$.
These are the standard estimates used in Section~\ref{sec:HeatErrAnalSpaceTime},
and the second-order accuracy is confirmed by the smooth test problems mentioned in Section~\ref{sec:num_results}.
On general admissible meshes, the available error estimates for the two-point flux approximation
are in general of first order only \cite{EymardGallouetHerbin2000}.
Note that the $L^2(\Omega)$ norm of $u-u_h$ for the piecewise constant function $u_h$
additionally contains the projection error $\|u-P_hu\|_{L^2(\Omega)}=O(h)$;
therefore, second-order statements for smooth solutions refer to the discrete norm $\|P_hu-u_h\|_h$,
equivalently, to the errors in the cell centers.
For the singular solutions in Sections~\ref{sec:PoisErrAnal} and~\ref{sec:HeatErrAnalPersistent},
the error bounds hold both for $\|u-u_h\|_{L^2(\Omega)}$ and for $\|P_hu-u_h\|_h\le\|u-u_h\|_{L^2(\Omega)}$.
The implementation uses at the boundary a formally first-order accurate discretization;
this is reasonable in view of the homogeneous Dirichlet boundary conditions,
and the numerical tests conform this in Section~\ref{subsec:stationary}.
\end{revblock}



\section{Finite Volume Error for a Poisson Equation with a Dirac Source}
\label{sec:PoisErrAnal}

\subsection{The two-dimensional case}\label{subsec:Pois2D}

We study the Poisson equation with a Dirac measure right-hand side
\begin{equation}
-\Delta u = \delta_{\xx_0} \quad \text{in } \Omega, 
\qquad 
u=0 \quad \text{on } \partial\Omega,
\label{poisson}
\end{equation}
where $\Omega\subset\mathbb{R}^2$ is a bounded domain and $\xx_0\in\Omega$.

Because the right-hand side is a measure, the solution is singular and classical
finite element or finite volume convergence rates are not valid.

Let $G(\cdot,\xx_0)$ denote the Green function of the Laplacian in $\Omega$ with homogeneous Dirichlet boundary conditions. 
In two dimensions the Green function has the form
\begin{equation}
G(\xx,\xx_0) = -\frac{1}{2\pi}\log|\xx-\xx_0| + w(\xx),
\end{equation}
where $w$ is   smooth in a neighborhood of $\xx_0$.

Therefore the solution of \eqref{poisson} can be written as
\begin{equation}
u(\xx) = -\frac{1}{2\pi}\log|\xx-\xx_0| + w(\xx).
\end{equation}

Let $\mathcal{T}_h$ be a quasi-uniform mesh of $\Omega$ with mesh size $h$, and let $u_h$ denote \rev{the finite volume
approximation of \eqref{poisson} by the two-point flux scheme of Section~\ref{sec:num_method},
identified with a piecewise constant function on $\mathcal{T}_h$ (see Subsection~\ref{subsec:FVnotation}).
We first present a heuristic scaling argument that identifies the sources of the error.
It treats $u_h$ as a local approximation of $u$ on each cell;
a rigorous proof, which also accounts for the propagation of local errors through $\Omega$
by the globally coupled discrete problem, is given in Subsection~\ref{subsec:PoisRigorous}.}

Let $K_0$ be the mesh cell containing $\xx_0$. Since the mesh is quasi-uniform,
\[
\operatorname{diam}(K_0)\sim h, 
\qquad |K_0|\sim h^2.
\]

We split the error as
\begin{equation}
\|u-u_h\|_{L^2(\Omega)}^2
=
\|u-u_h\|_{L^2(K_0)}^2
+
\|u-u_h\|_{L^2(\Omega\setminus K_0)}^2.
\end{equation}

We first estimate the $L^2$ norm of the singular part on a ball
\[
B_h(\xx_0)=\{\xx:|\xx-\xx_0|<h\}.
\]

Using polar coordinates,
\begin{align}
\int_{B_h(\xx_0)} |\log|\xx-\xx_0||^2\,d\xx
&=
2\pi\int_0^h (\log r)^2 r\,dr.
\end{align}

A primitive of $r(\log r)^2$ is
\[
\frac{r^2}{2}(\log r)^2
-\frac{r^2}{2}\log r
+\frac{r^2}{4}.
\]

Hence
\begin{equation}
\int_0^h (\log r)^2 r\,dr
=
\frac{h^2}{2}(\log h)^2
-\frac{h^2}{2}\log h
+\frac{h^2}{4}.
\end{equation}

As $h\to0$ the dominant term is
\[
\frac{h^2}{2}(\log h)^2.
\]

Therefore
\begin{equation}
\|u\|_{L^2(B_h(\xx_0))}\sim h|\log h|.
\end{equation}

\begin{revblock}
The size $h|\log h|$ of $u$ on $B_h(\xx_0)$ does not by itself determine the error on $K_0$,
because the constant value $u_{K_0}$ of the discrete solution on $K_0$ need not be small.
Denoting by $\ou_{K_0}=|K_0|^{-1}\int_{K_0}u\,d\xx$ the mean value of $u$ on $K_0$, we split
\begin{equation}\label{eq:nearfield-split}
\|u-u_h\|_{L^2(K_0)}
\le
\|u-\ou_{K_0}\|_{L^2(K_0)} + |K_0|^{1/2}\,|\ou_{K_0}-u_{K_0}|.
\end{equation}
The first term is of order $h$. Indeed, since $\log|\xx-\xx_0| = \log h + \log(|\xx-\xx_0|/h)$,
the logarithmic singularity minus the constant $\log h$ is invariant under the scaling
$\xx-\xx_0 = h\,\hat\xx$, and hence
\[
\|u-\ou_{K_0}\|_{L^2(K_0)}
\le
\frac{1}{2\pi}\,\big\|\log(|\xx-\xx_0|/h)\big\|_{L^2(K_0)} + C h^2
=
\frac{h}{2\pi}\,\big\|\log|\hat\xx|\big\|_{L^2(\hat K_0)} + Ch^2
\le C h,
\]
where $\hat K_0 = h^{-1}(K_0-\xx_0)$ has a diameter of order one.
In particular, the large constant $\frac{1}{2\pi}|\log h|$, which is responsible for the size
$h|\log h|$ of $u$ on $B_h(\xx_0)$, does not contribute to this term.
The second term in \eqref{eq:nearfield-split} measures how well the discrete solution reproduces
the mean value of $u$ on the singular cell. It is a genuinely discrete quantity that cannot be inferred
from properties of $u$ alone, since $u_{K_0}$ depends on the solution of the globally coupled discrete problem.
It is estimated rigorously in Subsection~\ref{subsec:PoisRigorous} 
with the result $|K_0|^{1/2}|\ou_{K_0}-u_{K_0}| \le C h|\log h|^{1/2}$. Hence
\begin{equation}
\|u-u_h\|_{L^2(K_0)} \lesssim h|\log h|^{1/2},
\end{equation}
which is smaller than the size $h|\log h|$ of $u$ itself on $B_h(\xx_0)$:
the discrete solution captures the logarithmic growth of $u$ in the singular cell to leading order.
\end{revblock}

Away from $\xx_0$, the solution is smooth. 

For a cell $K$ located at distance $r$ from $\xx_0$ we have
\[
|\nabla^2u(\xx)|\sim \frac{1}{r^2}.
\]

Standard interpolation estimates give\rev{, if $u_h$ is heuristically regarded as a local
second-order approximation of $u$ on $K$,}
\[
\|u-u_h\|_{L^2(K)} \lesssim h^2 |\nabla^2u| |K|^{1/2}.
\]

Since $|K|^{1/2}\sim h$ this yields
\[
\|u-u_h\|_{L^2(K)}\sim \frac{h^3}{r^2}.
\]

Summing contributions over annuli around $\xx_0$ gives a total error
\begin{equation}
\|u-u_h\|_{L^2(\Omega\setminus K_0)}\lesssim h.
\end{equation}

\begin{revblock}
This heuristic far-field bound ignores that $u_h$ is not a local interpolant of $u$:
the local truncation errors, in particular the flux errors of order one at the faces of $K_0$,
are propagated over $\Omega\setminus K_0$ by the discrete Green's function.
Moreover, for the piecewise constant function $u_h$, the local error on a cell $K$
is at best of order $h|\nabla u|\,|K|^{1/2}$, which after summation over the annuli yields $h|\log h|^{1/2}$ instead of $h$.
Both effects are taken into account in Theorem~\ref{thm:Poisson} in Subsection~\ref{subsec:PoisRigorous},
which shows that the overall error is of order $h|\log h|^{1/2}$, see \eqref{equ:ResultPois2D}.
\end{revblock}

\subsection{The three-dimensional case}\label{subsec:Pois3D}

We now consider the analogous problem in three dimensions:
\begin{equation}
-\Delta u = \delta_{\xx_0}\quad \text{in } \Omega\subset\mathbb{R}^3,
\qquad
u=0\quad \text{on } \partial\Omega.
\end{equation}

In three dimensions the Green function has the form
\begin{equation}
u(\xx)=\frac{1}{4\pi |\xx-\xx_0|}+w(\xx),
\end{equation}
where $w$ is smooth near $\xx_0$. Thus the singular part is
\[
g(\xx)=\frac{1}{4\pi r},\qquad r=|\xx-\xx_0|.
\]

Let $K_0$ be the mesh cell containing $\xx_0$, with diameter $h$. Comparing $K_0$
with a ball $B_h(\xx_0)$, we obtain
\[
\|g\|_{L^2(B_h(\xx_0))}^2
\sim
\int_{B_h(\xx_0)} \frac1{r^2}\,d\xx.
\]
Using spherical coordinates,
\[
\int_{B_h(\xx_0)} \frac1{r^2}\,d\xx
=
4\pi\int_0^h \frac1{r^2}r^2\,dr
=
4\pi\int_0^h dr
\sim h.
\]
Therefore
\begin{equation}
  \|g\|_{L^2(B_h(\xx_0))}\sim h^{1/2}.
\end{equation}

\rev{We use again the splitting \eqref{eq:nearfield-split}.
In contrast to the two-dimensional case, the first term is now of the same order as $\|g\|_{L^2(B_h(\xx_0))}$:
under the scaling $\xx-\xx_0=h\hat\xx$, the singular part becomes $g = h^{-1}\hat g(\hat\xx)$ with $\hat g(\hat\xx)=1/(4\pi|\hat\xx|)$,
and since the non-constant function $\hat g$ cannot be approximated by constants on $\hat K_0$,
$\|u-\ou_{K_0}\|_{L^2(K_0)} \sim h^{-1}h^{3/2} = h^{1/2}$.
The second term is again estimated in Subsection~\ref{subsec:PoisRigorous} (Remark~\ref{rem:nearfield})
and is at most of order $h^{1/2}$.
Since $u_h$ is constant on $K_0$, the first term is also a lower bound for the error on $K_0$, and thus}
\begin{equation}
\|u-u_h\|_{L^2(K_0)}\sim h^{1/2}.
\end{equation}

Away from $\xx_0$, the solution is smooth and
\[
|\nabla^2u(x)|\sim \frac1{r^3}.
\]
Let $K$ be a mesh cell at distance $r$ from $x_0$. Then a standard local estimate gives\rev{, again
regarding $u_h$ heuristically as a local second-order approximation of $u$,}
\[
\|u-u_h\|_{L^2(K)} \lesssim h^2 |\nabla^2u|\, |K|^{1/2}.
\]
Since in three dimensions $|K|^{1/2}\sim h^{3/2}$, it follows that
\[
\|u-u_h\|_{L^2(K)}
\lesssim
h^2 \frac1{r^3} h^{3/2}
=
\frac{h^{7/2}}{r^3}.
\]
Squaring,
\[
\|u-u_h\|_{L^2(K)}^2
\lesssim
\frac{h^7}{r^6}.
\]

Now consider a spherical shell
\[
A(r,r+dr)=\{x: r<|x-x_0|<r+dr\}.
\]
Its volume is approximately
\[
4\pi r^2\,dr.
\]
Since each cell has volume $\sim h^3$, the number of cells in the shell is
\[
N(r)\sim \frac{4\pi r^2\,dr}{h^3}.
\]
Therefore the shell contributes
\[
dE_{\mathrm{far}}^2(r)
\sim
N(r)\frac{h^7}{r^6}
=
\frac{4\pi r^2\,dr}{h^3}\frac{h^7}{r^6}
\sim
h^4\frac{dr}{r^4}.
\]
Integrating from $r\sim h$ to $r\sim 1$ gives
\[
E_{\mathrm{far}}^2
\sim
h^4 \int_h^1 r^{-4}\,dr.
\]
Since
\[
\int r^{-4}\,dr = -\frac{1}{3r^3},
\]
we obtain
\[
E_{\mathrm{far}}^2
\sim
h^4\left(\frac{1}{3h^3}-\frac13\right)
\sim h.
\]
Hence
\begin{equation}\label{equ:ResultPois3D1}
  \|u-u_h\|_{L^2(\Omega\setminus K_0)}\lesssim h^{1/2}.
\end{equation}

Both the near-field and far-field contributions are of the same order:
\[
\|u-u_h\|_{L^2(K_0)}\sim h^{1/2},
\qquad
\|u-u_h\|_{L^2(\Omega\setminus K_0)}\lesssim h^{1/2}.
\]
\rev{These heuristic scaling arguments therefore suggest}
\[
\|u-u_h\|_{L^2(\Omega)}\rev{\lesssim} h^{1/2},
\]
\rev{which is proved rigorously in Theorem~\ref{thm:Poisson} below; as in two dimensions,
the rigorous proof also accounts for the global propagation of the local errors.}

\begin{revblock}
\subsection{Rigorous error bound including the global propagation of local errors}
\label{subsec:PoisRigorous}

The scaling arguments of Subsections~\ref{subsec:Pois2D} and~\ref{subsec:Pois3D}
treat $u_h$ as a local approximation of $u$ on each cell.
Since $u_h$ is the solution of a globally coupled linear system,
local truncation errors---in particular the flux errors of order one
at the faces of the singular cell $K_0$---are propagated through $\Omega$
by the discrete solution operator, i.e., by the discrete Green's function.
We now give a proof that takes this propagation into account.
We work in the setting of the numerical experiments of Section~\ref{sec:num_results}
and use the notation of Subsection~\ref{subsec:FVnotation}.

\medskip
\noindent\textbf{Assumption (A).}
$\Omega\subset\R^d$, $d\in\{2,3\}$, is a rectangle or a rectangular box,
$\mathcal{T}_h$ is a uniform Cartesian mesh of $\Omega$ consisting of squares or cubes of side length $h$,
the scheme is the two-point flux approximation \eqref{eq:TPFA} with the cell centers as points $\xx_K$
(so that $d_e=h$ for interior faces and $d_e=h/2$ for boundary faces; the vertex-centered variant can be treated in the same way),
and there is a constant $c_0>0$ independent of $h$ such that
$\operatorname{dist}(\xx_0,\partial K_0)\ge c_0 h$ (for instance, $\xx_0$ is the center of $K_0$).

\medskip
For a face $e\in\mathcal{E}$ and a control volume $K$ with $e\subset\partial K$,
let $\nn_{K,e}$ be the unit normal on $e$ pointing out of $K$.
We compare the exact fluxes $\Phi_{K,e}(u) := -\int_e \nabla u\cdot\nn_{K,e}\,ds$
with the discrete fluxes $\Phi^h_{K,e}(v) := \frac{|e|}{d_e}(v_K - v_L) = -F_{K,e}(v)$,
where $L$ is the neighbor of $K$ across $e$ and $v_L:=0$ for $e\subset\partial\Omega$,
and define the flux consistency errors
\begin{equation}\label{eq:fluxerror}
\tau_{K,e} := \Phi^h_{K,e}(P_h u) - \Phi_{K,e}(u).
\end{equation}
Both fluxes are conservative, so that $\tau_{L,e}=-\tau_{K,e}$,
and we write $\tau_e := \tau_{K,e}$ with the orientation used in the definition of $D_e$.

\begin{lemma}[Error representation]\label{lem:errrep}
Let Assumption~(A) hold, let $u$ be the solution of \eqref{poisson},
and let $u_h$ be its finite volume approximation, $-\Delta_h u_h = \mu_h$.
Then $e_h := P_h u - u_h$ satisfies
\begin{equation}\label{eq:errrep}
a_h(e_h, w) = \sum_{e\in\mathcal{E}} \tau_e\, D_e w
\qquad\text{for all grid functions } w.
\end{equation}
\end{lemma}

\begin{proof}
Since $\xx_0\notin\partial K$ for all $K\in\mathcal{T}_h$, the solution $u$ is smooth in a neighborhood of each face,
and the divergence theorem together with $-\Delta u=\delta_{\xx_0}$ gives the exact flux balance
$\sum_{e\subset\partial K}\Phi_{K,e}(u) = 1$ if $K=K_0$ and $=0$ otherwise.
The scheme reads $\sum_{e\subset\partial K}\Phi^h_{K,e}(u_h) = |K|\,(\mu_h)_K$, which is the same right-hand side.
Subtracting and using the linearity of $\Phi^h_{K,e}$ yields
$\sum_{e\subset\partial K}\Phi^h_{K,e}(e_h) = \sum_{e\subset\partial K}\tau_{K,e}$ for every $K$.
Multiplying by $w_K$, summing over $K$, and collecting the two contributions of each interior face gives \eqref{eq:errrep}.
\end{proof}

\begin{lemma}[Flux consistency]\label{lem:consistency}
Let Assumption~(A) hold. Then there are constants $C$ and $h_0>0$ independent of $h$ such that for $h\le h_0$
\begin{equation}\label{eq:tau-bounds}
|\tau_e| \le
\begin{cases}
C, & \operatorname{dist}(e,\xx_0) < 2h,\\
C\,h^{d+1}\operatorname{dist}(e,\xx_0)^{-d-1}, & \operatorname{dist}(e,\xx_0)\ge 2h,\ e\not\subset\partial\Omega,\\
C\,h^{d}, & e\subset\partial\Omega,
\end{cases}
\;\text{and hence}\;
\sum_{e\in\mathcal{E}}|\tau_e| \le C.
\end{equation}
\end{lemma}

\begin{proof}
By odd reflection across the sides of $\Omega$, the solution $u$ is smooth on $\overline\Omega\setminus\{\xx_0\}$,
and $|\nabla^k u(\xx)|\le C|\xx-\xx_0|^{2-d-k}$ for $k=1,2,3$.
Let $e$ be an interior face between $K$ and $L$ with $\nn_{K,e}=\vec{e}_i$.
A direct computation shows that the discrete flux density is a weighted average of $\partial_i u$ over $Z_e:=K\cup L$,
\[
\frac{(P_hu)_L-(P_hu)_K}{h} = \frac{1}{|e|}\int_{e}\int_{-h}^{h}\phi_h(\sigma)\,\partial_i u(\yy+\sigma\vec{e}_i)\,d\sigma\,ds(\yy),
\qquad \phi_h(\sigma)=\frac{h-|\sigma|}{h^2},
\]
with the symmetric weight $\phi_h\ge0$, $\int_{-h}^h\phi_h\,d\sigma=1$, whereas the exact flux density is $|e|^{-1}\int_e\partial_iu\,ds$.
If $\operatorname{dist}(e,\xx_0)\ge 2h$, then $|\xx-\xx_0|\ge\operatorname{dist}(e,\xx_0)/2$ on $Z_e$;
a Taylor expansion in $\sigma$, in which the linear term vanishes by the symmetry of $\phi_h$, gives
$|\tau_e|\le C|e|\,h^2\sup_{Z_e}|\nabla^3u|\le Ch^{d+1}\operatorname{dist}(e,\xx_0)^{-d-1}$.
If $\operatorname{dist}(e,\xx_0)<2h$, then $e$ does not intersect the open cell $K_0$, so that
$\operatorname{dist}(e,\xx_0)\ge\operatorname{dist}(\xx_0,\partial K_0)\ge c_0h$ and
$|\Phi_{K,e}(u)|\le|e|\sup_e|\nabla u|\le Ch^{d-1}(c_0h)^{1-d}\le C$;
moreover, $Z_e\subset B_{5h}(\xx_0)$ and $\phi_h\le h^{-1}$ give
$|\Phi^h_{K,e}(P_hu)|\le h^{-1}\int_{Z_e}|\nabla u|\,d\xx\le Ch^{-1}\int_{B_{5h}(\xx_0)}|\xx-\xx_0|^{1-d}\,d\xx\le C$.
For a boundary face $e\subset\partial K\cap\partial\Omega$, a Taylor expansion in the normal direction, using $u=0$ on $e$
and $d_e=h/2$, gives $|\tau_e|\le C|e|\,h\sup_K|\nabla^2u|\le Ch^d$, since $K$ stays away from $\xx_0$ for $h\le h_0$.
Finally, there are at most $C$ faces with $\operatorname{dist}(e,\xx_0)<2h$,
at most $Ck^{d-1}$ interior faces with $kh\le\operatorname{dist}(e,\xx_0)<(k+1)h$, and at most $Ch^{1-d}$ boundary faces.
Hence $\sum_e|\tau_e|\le C\big(1+\sum_{k\ge2}k^{d-1}k^{-d-1}+h\big)\le C$.
\end{proof}

\begin{lemma}[Stability of the discrete dual problem]\label{lem:dual}
Let Assumption~(A) hold, let $\psi$ be a grid function, and let $z_h$ be the solution of
$a_h(w,z_h) = (\psi,w)_h$ for all grid functions $w$, i.e., $-\Delta_h z_h=\psi$. Then
\begin{equation}\label{eq:dualstab}
\max_{e\in\mathcal{E}}|D_e z_h| \le C\,h\,\ell_h\,\|\psi\|_h,
\qquad
\ell_h :=
\begin{cases}
(1+|\log h|)^{1/2}, & d=2,\\
h^{-1/2}, & d=3,
\end{cases}
\end{equation}
with $C$ independent of $h$ and $\psi$.
\end{lemma}

\begin{proof}
Under Assumption~(A), the matrix of $-\Delta_h$ is the Kronecker sum
$\sum_{i=1}^d I\otimes\cdots\otimes T_i\otimes\cdots\otimes I$ of symmetric positive definite tridiagonal matrices $T_i$,
one for each coordinate direction.
Since these summands commute, expanding $\|\Delta_h z_h\|_h^2$ yields
$\|(I\otimes\cdots\otimes T_i\otimes\cdots\otimes I)z_h\|_h\le\|\Delta_h z_h\|_h=\|\psi\|_h$,
and similarly for the mixed terms involving $T_i^{1/2}\otimes T_j^{1/2}$, $i\neq j$.
In terms of difference quotients, this is a discrete $H^2$-regularity estimate:
all second difference quotients of $z_h$, including the mixed ones, are bounded by $C\|\psi\|_h$ in the discrete $L^2$ norm.
Consequently, each first difference quotient $\partial_{h,i}z_h$,
regarded as a grid function on the faces normal to the $i$-th coordinate direction,
has a discrete $H^1$ norm bounded by $C\|\psi\|_h$.
The discrete Sobolev inequality $\|v\|_{\infty}\le C(1+|\log h|)^{1/2}\|v\|_{1,h}$ in two dimensions \cite{Bramble1966},
and the discrete Sobolev inequality $\|v\|_{L^6_h}\le C\|v\|_{1,h}$ combined with the inverse estimate
$\|v\|_\infty\le h^{-1/2}\|v\|_{L^6_h}$ in three dimensions, give $\max|\partial_{h,i}z_h|\le C\ell_h\|\psi\|_h$.
Since $|D_e z_h| = h\,|\partial_{h,i}z_h|$ on the faces normal to the $i$-th direction
($h/2$ instead of $h$ on boundary faces), \eqref{eq:dualstab} follows.
\end{proof}

\begin{theorem}\label{thm:Poisson}
Let Assumption~(A) hold, let $u$ be the solution of \eqref{poisson},
and let $u_h$ be its finite volume approximation, identified with a piecewise constant function.
Then there are constants $C$ and $h_0>0$ independent of $h$ such that for all $h\le h_0$
\begin{alignat}{2}
\|u-u_h\|_{L^2(\Omega)} &\le C\,h\,|\log h|^{1/2}, &\qquad& d=2, \label{equ:ResultPois2D}\\
\|u-u_h\|_{L^2(\Omega)} &\le C\,h^{1/2}, && d=3. \label{equ:ResultPois3D}
\end{alignat}
The same bounds hold for the discrete error $\|P_hu-u_h\|_h$.
\end{theorem}

\begin{proof}
We split $u-u_h = (u-P_hu) + e_h$ with $e_h = P_hu-u_h$, and note that
$h\le h\ell_h$, $h\ell_h = h^{1/2}$ for $d=3$, and $h\ell_h\le C h|\log h|^{1/2}$ for $d=2$ and $h\le h_0<1$.

\emph{Projection error.} For cells $K$ with $\operatorname{dist}(K,\xx_0)\ge h$,
the Poincar\'e inequality on $K$ gives $\|u-(P_hu)_K\|_{L^2(K)}\le Ch\|\nabla u\|_{L^2(K)}$, and summation yields
\[
\sum_{\operatorname{dist}(K,\xx_0)\ge h}\|u-(P_hu)_K\|_{L^2(K)}^2
\le C h^2\int_{\Omega\cap\{|\xx-\xx_0|\ge h\}}|\xx-\xx_0|^{2-2d}\,d\xx
\le C h^2\ell_h^2 .
\]
For each of the boundedly many remaining cells, which lie in $B_{3h}(\xx_0)$,
the scaling argument used for the first term in \eqref{eq:nearfield-split} gives
$\|u-(P_hu)_K\|_{L^2(K)}\le Ch$ for $d=2$, and
$\|u-(P_hu)_K\|_{L^2(K)}\le\|u\|_{L^2(K)}\le C\big\||\xx-\xx_0|^{-1}\big\|_{L^2(B_{3h}(\xx_0))}+Ch^{3/2}\le Ch^{1/2}$ for $d=3$.
Hence $\|u-P_hu\|_{L^2(\Omega)}\le Ch\ell_h$.

\emph{Discrete error.} Let $z_h$ be the solution of the dual problem of Lemma~\ref{lem:dual} with $\psi=e_h$.
By Lemmas~\ref{lem:errrep}, \ref{lem:consistency} and~\ref{lem:dual},
\[
\|e_h\|_h^2 = a_h(e_h,z_h) = \sum_{e\in\mathcal{E}}\tau_e\,D_e z_h
\le \Big(\sum_{e\in\mathcal{E}}|\tau_e|\Big)\max_{e\in\mathcal{E}}|D_e z_h|
\le C\,h\,\ell_h\,\|e_h\|_h .
\]
Hence $\|P_hu-u_h\|_h=\|e_h\|_h \le C h\ell_h$, and the triangle inequality completes the proof.
\end{proof}

\begin{remark}[Error on the singular cell]\label{rem:nearfield}
Since $|K_0|^{1/2}|\ou_{K_0}-u_{K_0}| = \|e_h\|_{L^2(K_0)}\le\|e_h\|_h$,
the proof yields $|K_0|^{1/2}|\ou_{K_0}-u_{K_0}|\le Ch\ell_h$,
which is the bound for the second term in \eqref{eq:nearfield-split}
used in Subsections~\ref{subsec:Pois2D} and~\ref{subsec:Pois3D}.
For $d=2$, this means $|\ou_{K_0}-u_{K_0}|\le C|\log h|^{1/2}$, whereas $\ou_{K_0}=\frac{1}{2\pi}|\log h|+O(1)$:
the discrete solution reproduces the logarithmic growth of $u$ in the singular cell to leading order.
For $d=3$, the error on $K_0$ is bounded from below by $\|u-\ou_{K_0}\|_{L^2(K_0)}\ge c\,h^{1/2}$,
so that the rate in \eqref{equ:ResultPois3D} cannot be improved.
\end{remark}

\begin{remark}[Mechanism of the order reduction]\label{rem:mechanism}
The proof makes the propagation of local errors explicit.
By Lemma~\ref{lem:consistency}, the far-field flux errors are summable,
while the finitely many faces near $\xx_0$ carry flux errors of order one, which do not decrease under mesh refinement.
Through \eqref{eq:errrep}, these errors are propagated by the discrete dual solution, i.e., by the discrete Green's function,
and they enter the global error multiplied by the discrete gradients $D_ez_h$ near $\xx_0$.
The size $h\ell_h$ of these gradients reflects that the dual solution is only in $H^2$,
which does not embed into $W^{1,\infty}$ for $d=2,3$.
This is the mechanism that reduces the convergence order from the nominal value $2$ to $2-d/2$
(up to a logarithmic factor for $d=2$), in analogy with the finite element result of Scott \cite{Scott1973},
who proved the $L^2$ error bound $O(h)$ for piecewise linear finite elements in two dimensions.
\end{remark}

\begin{remark}[Extensions]\label{rem:extensions}
Assumption~(A) enters only through the summability of the flux errors and the dual stability \eqref{eq:dualstab}.
On general quasi-uniform admissible meshes, the two-point flux approximation is in general only first-order consistent
\cite{EymardGallouetHerbin2000}, so that the far-field flux errors are of order $h^d\operatorname{dist}(e,\xx_0)^{-d}$
and their sum grows like $|\log h|$; provided a dual stability estimate like \eqref{eq:dualstab} is available,
the argument then yields the same rates up to an additional logarithmic factor.
Moreover, if the dual problem with data $\psi\in L^\infty(\Omega)$ satisfies the discrete $W^{1,\infty}$ estimate
$\max_e|D_ez_h|\le Ch\|\psi\|_{L^\infty(\Omega)}$, which for the five- and seven-point schemes
is a consequence of discrete $W^{2,p}$-regularity with $p>d$,
the same duality argument in $L^1$ gives $\|P_hu-u_h\|_{L^1(\Omega)}\le Ch$.
Since $\|u-P_hu\|_{L^1(\Omega)}\le Ch\|\nabla u\|_{L^1(\Omega)}$ with $\nabla u\in L^1(\Omega)$,
this implies $\|u-u_h\|_{L^1(\Omega)}\le Ch$ for $d=2,3$.
\end{remark}
\end{revblock}

\section{Finite Volume Error for a Heat Equation with a Space-Time Dirac Source}
\label{sec:HeatErrAnalSpaceTime}

We consider the heat equation
\begin{equation}
u_t-\Delta u=\delta_{(x_0,t_0)}
\qquad \text{in } \Omega \times (0,T)
\label{eq:heat-dirac}
\end{equation}
with homogeneous Dirichlet boundary conditions
\[
u=0 \qquad \text{on } \partial\Omega \times (0,T),
\]
and vanishing initial condition
\[
u(x,0)=0.
\]

The right-hand side is a Dirac mass concentrated at the space-time point $(x_0,t_0)$.

\subsection{Exact solution in the whole space}

In $\mathbb{R}^d$, the fundamental solution of the heat equation is the heat kernel
\begin{equation}
G(\xx,t)=
(4\pi t)^{-d/2}\exp\!\left(-\frac{|\xx|^2}{4t}\right),
\qquad t>0.
\end{equation}
Therefore, in the whole space, the solution of \eqref{eq:heat-dirac} is
\begin{equation}
u(\xx,t)=
\begin{cases}
0, & t<t_0,\\[4pt]
G(\xx-\xx_0,t-t_0), & t>t_0.
\end{cases}
\label{eq:heat-kernel-shifted}
\end{equation}

Thus, if we set
\[
s:=t-t_0>0,
\]
then
\begin{equation}
u(\xx,t)
=
(4\pi s)^{-d/2}
\exp\!\left(-\frac{|\xx-\xx_0|^2}{4s}\right).
\end{equation}

This shows that the singularity is instantaneous in time: at $t=t_0$ the solution is measure-valued, but for every fixed $t>t_0$ the solution is smooth in space.

\subsection{Scaling of the solution}

The heat kernel has the self-similar form
\begin{equation}
u(\xx,t)=s^{-d/2}\,\Phi\!\left(\frac{\xx-\xx_0}{\sqrt{s}}\right),
\qquad
\Phi(\yy):=(4\pi)^{-d/2}e^{-|\yy|^2/4}.
\label{eq:selfsimilar}
\end{equation}

Differentiating with respect to $\xx$ gives
\begin{equation}
\nabla^m u(\xx,t)
=
s^{-d/2}s^{-m/2}
\Psi_m\!\left(\frac{\xx-\xx_0}{\sqrt{s}}\right),
\end{equation}
where $\Psi_m$ is a polynomial times a Gaussian \rev{distribution function.} Therefore
\begin{equation}
|\nabla^m u(\xx,t)|
\sim
s^{-d/2-m/2}
\times
\text{Gaussian factor}.
\end{equation}

Taking the $L^2$ norm and using the change of variables
\[
\yy=\frac{\xx-\xx_0}{\sqrt{s}},
\qquad d\xx=s^{d/2}\,d\yy,
\]
we obtain
\begin{align}
\|\nabla^m u(\cdot,t)\|_{L^2(\mathbb{R}^d)}^2
&\sim
\int_{\mathbb{R}^d}
s^{-d-m}
\left|\Psi_m\!\left(\frac{\xx-\xx_0}{\sqrt{s}}\right)\right|^2
d\xx
\\
&=
s^{-d-m}s^{d/2}
\int_{\mathbb{R}^d} |\Psi_m(\yy)|^2\,d\yy
\\
&\sim
s^{-m-d/2}.
\end{align}
Hence
\begin{equation}
\|\nabla^m u(\cdot,t)\|_{L^2(\mathbb{R}^d)}
\sim
s^{-m/2-d/4}.
\label{eq:derivative-scaling}
\end{equation}

In particular, for $m=2$,
\begin{equation}
|u(\cdot,t)|_{H^2}
\sim
s^{-1-d/4}.
\label{eq:h2-scaling}
\end{equation}

\subsection{Finite volume approximation}

Let $u_h(\cdot,t)$ be \rev{the semidiscrete approximation by the two-point flux finite volume method
of Section~\ref{sec:num_method}, i.e., $u_h'-\Delta_hu_h=0$ for $t>t_0$ with $u_h(t_0)=\mu_h$,}
on a quasi-uniform mesh of size $h$.
\rev{As explained in Subsection~\ref{subsec:FVnotation}, \emph{low-order} refers to this lowest-order scheme
with one unknown per control volume, piecewise constant approximations, and two-point fluxes without reconstruction.
For smooth solutions on the uniform meshes considered here, it is second-order accurate in the discrete $L^2$ norm
and first-order accurate in the discrete energy norm $\|\cdot\|_{\mathrm{energy}}=|\cdot|_{1,h}$,
with the $L^2$ error measured in the discrete sense of Subsection~\ref{subsec:FVnotation}.}

\rev{Accordingly}, for smooth functions, standard \rev{spatial error} estimates \rev{for this scheme} have the form
\begin{equation}
\|u(\cdot,t)-u_h(\cdot,t)\|_{L^2(\Omega)}
\lesssim
h^2\,|u(\cdot,t)|_{H^2(\Omega)},
\label{eq:l2-fv-smooth}
\end{equation}
and, in a discrete energy norm,
\begin{equation}
\|u(\cdot,t)-u_h(\cdot,t)\|_{\mathrm{energy}}
\lesssim
h\,|u(\cdot,t)|_{H^2(\Omega)}.
\label{eq:energy-fv-smooth}
\end{equation}

The point is that for every fixed $t>t_0$, the heat kernel is smooth, so these estimates apply. The only special feature is that the $H^2$ norm depends singularly on
\[
s=t-t_0.
\]

Substituting \eqref{eq:h2-scaling} into \eqref{eq:l2-fv-smooth} and \eqref{eq:energy-fv-smooth}, we obtain
\begin{equation}
\|u(\cdot,t)-u_h(\cdot,t)\|_{L^2(\Omega)}
\lesssim
h^2\,s^{-1-d/4},
\qquad s=t-t_0>0,
\label{eq:l2-final-general}
\end{equation}
and
\begin{equation}
\|u(\cdot,t)-u_h(\cdot,t)\|_{\mathrm{energy}}
\lesssim
h\,s^{-1-d/4}.
\label{eq:energy-final-general}
\end{equation}

\subsection{Explicit rates in two dimensions}

If $d=2$, then \eqref{eq:h2-scaling} becomes
\begin{equation}
|u(\cdot,t)|_{H^2}\sim s^{-3/2}.
\end{equation}
Therefore
\begin{equation}
\|u(\cdot,t)-u_h(\cdot,t)\|_{L^2(\Omega)}
\lesssim
h^2\,s^{-3/2},
\qquad s=t-t_0>0,
\label{eq:l2-2d-heat}
\end{equation}
and
\begin{equation}
\|u(\cdot,t)-u_h(\cdot,t)\|_{\mathrm{energy}}
\lesssim
h\,s^{-3/2}.
\label{eq:energy-2d-heat}
\end{equation}

Thus, in two dimensions, for every fixed positive time after the impulse, the spatial finite volume method recovers the standard rates:
\[
O(h^2)\quad\text{in }L^2,
\qquad
O(h)\quad\text{in energy norm}.
\]

\subsection{Explicit rates in three dimensions}

If $d=3$, then \eqref{eq:h2-scaling} becomes
\begin{equation}
|u(\cdot,t)|_{H^2}\sim s^{-7/4}.
\end{equation}
Therefore
\begin{equation} \label{equ:ResultHeatSpaceTime2D} 
\|u(\cdot,t)-u_h(\cdot,t)\|_{L^2(\Omega)}
\lesssim
h^2\,s^{-7/4},
\qquad s=t-t_0>0,
\end{equation}
and
\begin{equation}
\|u(\cdot,t)-u_h(\cdot,t)\|_{\mathrm{energy}}
\lesssim
h\,s^{-7/4}.
\label{eq:energy-3d-heat}
\end{equation}

For the heat equation
\[
u_t-\Delta u=\delta_{(\xx_0,\rev{t_0})},
\]
the low-order \rev{(two-point flux)} finite volume method satisfies, for every fixed $t>t_0$,
\begin{equation} \label{equ:ResultHeatSpaceTime3D}
\|u(\cdot,t)-u_h(\cdot,t)\|_{L^2(\Omega)}
\lesssim
\begin{cases}
h^2 (t-t_0)^{-3/2}, & d=2,\\[4pt]
h^2 (t-t_0)^{-7/4}, & d=3,
\end{cases}
\end{equation}
and
\begin{equation}
\|u(\cdot,t)-u_h(\cdot,t)\|_{\mathrm{energy}}
\lesssim
\begin{cases}
h (t-t_0)^{-3/2}, & d=2,\\[4pt]
h (t-t_0)^{-7/4}, & d=3.
\end{cases}
\end{equation}

\section{Finite Volume Error for a Heat Equation with a Persistent Dirac Source}
\label{sec:HeatErrAnalPersistent}

\begin{proposition}
Let $\Omega\subset\mathbb{R}^d$ ($d\ge2$) be a bounded domain and $\xx_0\in\Omega$.
Consider
\begin{equation}
\partial_t u - \Delta u = \delta_{\xx_0}\,\mathbf{1}_{t\ge t_0}
\quad \text{in } \Omega \times (0,T)
\qquad
u=0 \text{ on } \partial\Omega,
\qquad
u(0)=0.
\end{equation}

Let $\mathcal{T}_h$ be a quasi-uniform admissible finite volume mesh and let $K_0\in\mathcal{T}_h$ be the control volume containing $\xx_0$\rev{;
for the estimate of the discretization error (Step~5 of the proof), we assume that $\Omega$, $\mathcal{T}_h$, and $\xx_0$
satisfy Assumption~(A) of Subsection~\ref{subsec:PoisRigorous}}. Define the regularized source
\[
\mu_h := \frac{1}{|K_0|}\mathbf{1}_{K_0}.
\]

Let $u^{(h)}$ solve
\begin{equation}
\partial_t u^{(h)} - \Delta u^{(h)} = \mu_h\,\mathbf{1}_{t\ge t_0},
\qquad
u^{(h)}|_{\partial\Omega}=0,\qquad u^{(h)}(0)=0,
\end{equation}
and let $u_h(t)$ be the semidiscrete finite volume approximation\rev{,
$u_h'-\Delta_hu_h=\mu_h\mathbf{1}_{t\ge t_0}$, $u_h(0)=0$, see Subsection~\ref{subsec:FVnotation}}.

Then for every $\tau>0$ there exists $C>0$ such that for all $t\in[t_0+\tau,T]$,
\begin{align} \label{equ:ResultL1L2}
\|u(t)-u_h(t)\|_{L^1(\Omega)} &\le C h,\\
\|u(t)-u_h(t)\|_{L^2(\Omega)} &\le  \label{equ:ResultHeatPersistentSemi}
\begin{cases}
C h |\log h|^{1/2}, & d=2,\\
C h^{1/2}, & d=3.
\end{cases}
\end{align}

Moreover, for any compact $K\Subset \Omega\setminus\{\xx_0\}$,
\[
\|u(t)-u_h(t)\|_{L^p(K)} \le C (h^r),
\]
where $r$ is the order of the underlying FV scheme for smooth solutions.
\end{proposition}

\begin{proof}
We split the error as
\[
u - u_h = (u - u^{(h)}) + (u^{(h)} - u_h) =: \eta_h + \xi_h.
\]

\medskip
\noindent\textbf{Step 1: representation formula.}
Let $G_D(t, \xx,\yy)$ be the Dirichlet heat kernel. Then
\[
u(\xx,t) = \int_{t_0}^t G_D(t-s,\xx,\xx_0)\,ds,
\quad
u^{(h)}(\xx,t) = \int_{t_0}^t \frac{1}{|K_0|}\int_{K_0} G_D(t-s,\xx,\yy)\,d\yy\,ds.
\]
Hence
\[
\eta_h(\xx,t)
= \int_{t_0}^t \left(
G_D(t-s,\xx,\xx_0) - \frac{1}{|K_0|}\int_{K_0} G_D(t-s,\xx,\yy)\,d\yy
\right) ds.
\]

\medskip
\noindent\textbf{Step 2: pointwise estimate.}
\rev{For $|\xx-\xx_0|\ge 2h$,} by the mean value theorem in $\yy$,
\[
|\eta_h(\xx,t)|
\le C h \int_{t_0}^t \sup_{\yy\in K_0} |\nabla_\yy G_D(t-s,\xx,\yy)|\,ds.
\]
Using the classical estimate
\[
\int_0^\infty |\nabla_\yy G_D(\tau,\xx,\yy)|\,d\tau \lesssim |\xx-\xx_0|^{1-d},
\]
we obtain
\[
|\eta_h(\xx,t)| \lesssim h |\xx-\xx_0|^{1-d}\rev{,\qquad |\xx-\xx_0|\ge 2h}.
\]

\medskip
\noindent\textbf{Step 3: $L^1$ estimate.}
Integrating in polar coordinates,
\[
\|\eta_h(t)\|_{L^1(\Omega)} \lesssim h \int_h^1 r^{1-d} r^{d-1} dr \lesssim h.
\]

\medskip
\noindent\textbf{Step 4: $L^2$ estimate.}
Similarly,
\[
\|\eta_h(t)\|_{L^2(\Omega)}^2 \lesssim h^2 \int_h^1 r^{1-d} dr,
\]
which gives
\[
\|\eta_h(t)\|_{L^2} \lesssim 
\begin{cases}
h |\log h|^{1/2}, & d=2,\\
h^{1/2}, & d=3.
\end{cases}
\]
\begin{revblock}
Here the pointwise bound of Step~2 has been used for $|\xx-\xx_0|\ge 2h$ only.
In the remaining ball $B_{2h}(\xx_0)$, we use that $\int_0^\infty G_D(\sigma,\xx,\yy)\,d\sigma = G(\xx,\yy)$
is the Dirichlet Green's function of the Laplacian, so that Step~1 gives
$\eta_h(t) = b_h - E(t-t_0)b_h$ with $b_h(\xx) := G(\xx,\xx_0)-|K_0|^{-1}\int_{K_0}G(\xx,\yy)\,d\yy$,
where $E(s)$ denotes the Dirichlet heat semigroup.
For $s=t-t_0\ge\tau$, we have $\|E(s)b_h\|_{L^\infty(\Omega)}\le C(\tau)\|b_h\|_{L^1(\Omega)}\le C$,
and on $B_{2h}(\xx_0)$ we have $|b_h(\xx)|\le C(1+|\log(|\xx-\xx_0|/h)|)$ for $d=2$
(the constant $\frac{1}{2\pi}|\log h|$ cancels in $b_h$, as in \eqref{eq:nearfield-split})
and $|b_h(\xx)|\le C(|\xx-\xx_0|^{-1}+h^{-1})$ for $d=3$.
Hence the near field contributes at most $Ch$ for $d=2$ and $Ch^{1/2}$ for $d=3$ to the $L^2$ norm,
and $O(h^2)$ to the $L^1$ norm, so that the bounds of Steps~3 and~4 hold on all of $\Omega$.
\end{revblock}

\medskip
\noindent\textbf{Step 5: discretization error.}
\begin{revblock}
The function $u^{(h)}$ solves a heat equation with a bounded source, but this bound degenerates:
$\|\mu_h\|_{L^\infty(\Omega)}=|K_0|^{-1}\sim h^{-d}$, and the higher derivatives of $u^{(h)}$ blow up near $K_0$ as $h\to0$;
for instance, $\|u^{(h)}(t)\|_{H^2(\Omega)}\sim\|\mu_h\|_{L^2(\Omega)}\sim h^{-d/2}$.
Standard consistency estimates of the form $\|\xi_h(t)\|_{L^2(\Omega)}\le Ch\sup_{s\le t}\|u^{(h)}(s)\|_{H^2(\Omega)}$
would therefore only give the useless bound $O(h^{1-d/2})$.
We avoid derivatives of $u^{(h)}$ near $K_0$ altogether.
Since $\xi_h = (u-u_h)-\eta_h$ and $\eta_h$ has been estimated in Steps~3 and~4, it suffices to estimate $u-u_h$.

Let $u_\infty := G(\cdot,\xx_0)$ be the solution of the Poisson problem \eqref{poisson},
and let $u_{\infty,h}$ be its finite volume approximation, $-\Delta_h u_{\infty,h}=\mu_h$;
recall from Subsection~\ref{subsec:FVnotation} that the finite volume source vectors of $\delta_{\xx_0}$ and of $\mu_h$ coincide.
With the semigroup $E(s)$ of Step~4 and its semidiscrete counterpart $E_h(s):=e^{s\Delta_h}$,
Duhamel's formula gives for $s:=t-t_0\ge\tau$
\[
u(t) = u_\infty - E(s)u_\infty,
\qquad
u_h(t) = u_{\infty,h} - E_h(s)u_{\infty,h},
\]
and hence, with $v(t):=E(s)u_\infty$ and $\varepsilon_h(t):=P_hv(t)-E_h(s)u_{\infty,h}$,
\[
u(t)-u_h(t) = (u_\infty-u_{\infty,h}) - \big(v(t)-P_hv(t)\big) - \varepsilon_h(t).
\]

(a) The first term is the error of the stationary problem, $\|u_\infty-u_{\infty,h}\|_{L^2(\Omega)}\le Ch\ell_h$
by Theorem~\ref{thm:Poisson}. This is the only place where the singularity at $\xx_0$ enters,
and Theorem~\ref{thm:Poisson} controls it through the flux errors of order one at the finitely many faces near $\xx_0$
(Lemma~\ref{lem:consistency}) and the stability of the dual problem (Lemma~\ref{lem:dual}),
not through derivatives near $\xx_0$.

(b) For $s\ge\tau$, $v(t)$ is smooth, and $\|v(t)-P_hv(t)\|_{L^2(\Omega)}\le Ch\|\nabla v(t)\|_{L^2(\Omega)}\le C(\tau)h$.

(c) The function $\varepsilon_h$ starts from $\varepsilon_h(t_0)=P_hu_\infty-u_{\infty,h}$,
with $\|\varepsilon_h(t_0)\|_h\le Ch\ell_h$ by the proof of Theorem~\ref{thm:Poisson},
and satisfies $\varepsilon_h'-\Delta_h\varepsilon_h=\rho_h$, where, as in Lemma~\ref{lem:errrep},
$(\rho_h(t),w)_h=\sum_{e}\tau_e(t)\,D_ew$ with the flux consistency errors $\tau_e(t)$ of $v(t)$.
The standard first-order flux consistency of the two-point scheme \cite{EymardGallouetHerbin2000} gives
$\|\rho_h(t)\|_{-1,h}\le Ch\,|v(t)|_{H^2(\Omega)}$, where $\|\cdot\|_{-1,h}$ is the norm dual to $|\cdot|_{1,h}$.
Since $u_\infty\in H^{1/4}(\Omega)$ (in fact, $u_\infty\in H^\beta(\Omega)$ for all $\beta<2-d/2$),
parabolic smoothing and $H^2$-regularity on the convex domain $\Omega$ yield
$|v(t)|_{H^2(\Omega)}\le C\|\Delta v(t)\|_{L^2(\Omega)}\le C (t-t_0)^{-7/8}\|u_\infty\|_{H^{1/4}(\Omega)}$.
Using that $E_h(s)$ is a contraction in $\|\cdot\|_h$ and the discrete smoothing property
$\|E_h(s)\rho\|_h\le Cs^{-1/2}\|\rho\|_{-1,h}$, which follows from the spectral representation of the symmetric operator $\Delta_h$,
Duhamel's formula gives
\[
\|\varepsilon_h(t)\|_h
\le \|\varepsilon_h(t_0)\|_h + Ch\int_0^{s}(s-\sigma)^{-1/2}\sigma^{-7/8}\,d\sigma
\le Ch\ell_h + C(\tau)h .
\]

Combining (a)--(c) shows $\|u(t)-u_h(t)\|_{L^2(\Omega)}\le Ch\ell_h$ for $t\in[t_0+\tau,T]$, and thus
$\|\xi_h(t)\|_{L^2(\Omega)}\le\|u(t)-u_h(t)\|_{L^2(\Omega)}+\|\eta_h(t)\|_{L^2(\Omega)}\le Ch\ell_h$,
which is the asserted order.
For the $L^1$ bound, the stationary errors in (a) and in $\varepsilon_h(t_0)$ are estimated in $L^1(\Omega)$
as in Remark~\ref{rem:extensions}, the propagation of $\varepsilon_h(t_0)$ is controlled since $E_h(s)$ is a contraction
in the discrete $L^1$ norm (the matrix of $-\Delta_h$ is an M-matrix, and the scheme is conservative),
and the remaining terms are bounded by $Ch$ in $L^2(\Omega)$ and hence in $L^1(\Omega)$.
\end{revblock}
Standard finite volume consistency yields optimal rates on compact subsets $K\Subset \Omega\setminus\{\xx_0\}$\rev{, which gives the local estimate}.

\medskip
Combining both terms gives the result.
\end{proof}

\begin{proposition}[Fully discrete FV error with backward Euler]
Under the assumptions of the previous proposition, consider time steps $\Delta t>0$ and $t_n = n\Delta t$.

Let $u_h^n$ be defined by the backward Euler finite volume scheme
\begin{equation}
\frac{u_h^n - u_h^{n-1}}{\Delta t} - \Delta_h u_h^n = f_h^n,
\end{equation}
with
\[
f_h^n =
\begin{cases}
0, & t_n < t_0,\\
\mu_h, & t_n \ge t_0.
\end{cases}
\]

Then for every $\tau>0$ there exists $C>0$ such that for all $t_n\in[t_0+\tau,T]$,
\begin{align}
\|u(t_n)-u_h^n\|_{L^1(\Omega)} &\le C (h + \Delta t),\\
\|u(t_n)-u_h^n\|_{L^2(\Omega)} &\le  \label{equ:ResultHeatPersistent}
\begin{cases}
C (h |\log h|^{1/2} + \Delta t), & d=2,\\
C (h^{1/2} + \Delta t), & d=3.
\end{cases}
\end{align}

Moreover, for any $K\Subset \Omega\setminus\{\xx_0\}$,
\[
\|u(t_n)-u_h^n\|_{L^p(K)} \le C (h^r + \Delta t),
\]
where $r$ is the spatial order of the FV scheme for smooth solutions.
\end{proposition}

\begin{proof}
Let $U_h(t)$ denote the semidiscrete solution. Then
\[
u(t_n)-u_h^n
=
(u-u^{(h)})(t_n)
+
(u^{(h)}(t_n)-U_h(t_n))
+
(U_h(t_n)-u_h^n).
\]

\medskip
\noindent\textbf{Step 1: source regularization.}
From the semidiscrete analysis,
\[
\|u(t_n)-u^{(h)}(t_n)\|_{L^1} \lesssim h,
\quad
\|u(t_n)-u^{(h)}(t_n)\|_{L^2}
\lesssim
\begin{cases}
h |\log h|^{1/2}, & d=2,\\
h^{1/2}, & d=3.
\end{cases}
\]

\medskip
\noindent\textbf{Step 2: spatial discretization.}
The term $u^{(h)}-U_h$ is estimated as in the semidiscrete case and does not exceed the same order.

\medskip
\noindent\textbf{Step 3: time discretization.}
Since $U_h$ solves a linear parabolic problem with \rev{a source that is constant for $t\ge t_0$}, standard backward Euler estimates give
\[
\|U_h(t_n)-u_h^n\| \le C \Delta t,
\]
uniformly for $t_n \ge t_0+\tau$\rev{, with $C$ independent of $h$.
Indeed, as in Step~5 of the previous proof, $U_h(t)=u_{\infty,h}-E_h(t-t_0)u_{\infty,h}$, and
$u_h^n = u_{\infty,h}-r(-\Delta t\,\Delta_h)^{k}u_{\infty,h}$ with $r(\lambda)=(1+\lambda)^{-1}$,
where $k$ is the number of time steps since the source has been switched on.
The nonsmooth-data estimate $\sup_{\lambda\ge0}|r(\lambda)^k-e^{-k\lambda}|\le C/k$ for the backward Euler method \cite{Thomee2006}
gives $\|U_h(t_n)-u_h^n\|_h\le C\,\Delta t\,(t_n-t_0)^{-1}\|u_{\infty,h}\|_h$,
up to a shift of the switching-on time by at most one step, which causes an additional error of the same order.
Since $\|u_{\infty,h}\|_h\le\|P_hu_\infty\|_h+\|P_hu_\infty-u_{\infty,h}\|_h\le C$ uniformly in $h$ by Theorem~\ref{thm:Poisson},
the factor $h^{-d}$ in $\mu_h$ does not enter the constant}.
\end{proof}



\section{Stationary and Time-Dependent Numerical Tests} \label{sec:num_results}

Analogous to the analysis, we present numerical demonstrations
for a stationary test case in Section~\ref{subsec:stationary}
and a time-dependent test case in Section~\ref{subsec:transient}.


\subsection{Stationary diffusion-reaction equation with homogeneous Dirichlet boundary conditions}
\label{subsec:stationary}

The stationary measure-valued test problem is given by
the diffusion-reaction portion of the time-dependent problem \eqref{eq:DREQ} 
\begin{align}
    - \nabla \cdot (D \, \nabla u) &= \kappa \, \delta_{\0}\!(\xx)
        &&\text{ for $\xx \in \Omega$,}              \\
    u(\xx) &= 0
        &&\text{ for $\xx \in \partial\Omega$,}
\end{align}
on the domain $\Omega = (-50,50)^d \subset \R^d$ in $d = 2, 3$ spatial dimensions.
\rev{This problem models a pollutant that diffuses with diffusivity $D = 40$ and is injected into the domain at a single location.}
The injection is modeled as a point source by the measure-valued
Dirac delta distribution $\delta_{\0}\!(\xx)$
located at the center $\0$ of the domain $\Omega$
with constant injection rate $\kappa = 10$.
Analogous to the analysis, the boundary condition is a homogeneous Dirichlet condition.

Here we present convergence studies for this test case.
For smooth forcing functions on the right-hand side of the stationary problem with an analytical solution $u$,
the expected second-order convergence \emph{uniformly for all spatial dimensions}
of the numerical solution $u_h$ on a discretization
with mesh size $h$ is verified by computing the true error $u - u_h$;
\begin{revblock}
the results are not included here since they are not novel.
They do serve to confirm that the formally first-order accurate boundary discretization
does not pollute the second-order convergence overall.
\end{revblock}
For problems, whose true solution $u$ is not available, it is customary
to use the numerical solution on the finest available mesh as a reference solution $u_{ref}$ in place of $u$ in the norm of the errors.
This procedure is necessary for this non-smooth test problem.
The error norm is computed in Matlab by interpolating the numerical solution on the coarser mesh
to the fine mesh of the reference solution before computing their difference.

Table~\ref{tab:PoissonDiracDirichlet} shows, for 2-D and 3-D,
the errors against the reference solution for several uniform refinement levels $r = 0, 1, 2, \ldots$.
In parentheses are listed the convergence order $p_{est}$ estimated as logarithm of the
ratio between the error norms of consecutive refinement levels
\begin{equation}
\label{equ:pest}
p_{est} = \log_2\left(\frac{\norm{u_{r-1} - u_{ref}}}{\norm{u_{r} - u_{ref}}} \right),
\quad r = 1, 2, \ldots .
\end{equation}
The result (\ref{equ:ResultPois2D}) in Section~\ref{sec:PoisErrAnal} leads us to expect
$O(h |\log h|^{\rev{1/2}}) \to O(h)$ as $h \to 0$ in  2-D and
(\ref{equ:ResultPois3D}) to expect $O(h^{1/2}) = O(h^{0.5})$ in 3-D.

We also show the convergence order estimate in the log-log-plot for 2-D and 3-D
in Figure~\ref{fig:PoissonDiracDirichlet}
as a linear fit of $\log{\|u-u_h\|}$ against $\log{1/h}$,
whose slope estimates the convergence order.

We observe in both Table~\ref{tab:PoissonDiracDirichlet} and Figure~\ref{fig:PoissonDiracDirichlet} that
the numerical convergence order is approaching $p = 1$ for 2-D while close to $p = 0.5$ in 3-D,
thus the values listed in the table and the trends shown in the figure are consistent with the
results presented in  Section~\ref{sec:PoisErrAnal} for 2-D and 3-D.
\begin{revblock}
Since the reference solution $u_{ref}$ is itself a
numerical solution with an error of the same structure, as $r$ approaches the reference level,
$u$ and $u_{ref}$ become increasingly correlated,
so $\|u-u_{ref}\|$ underestimates the true error and
the estimated order appears too large. The orders should therefore be judged from the coarser
levels.
\end{revblock}

%

\begin{table} \centering
  \caption{Stationary diffusion-reaction equation with homogeneous Dirichlet boundary conditions and Dirac source term:
           observed $L^2$-error against reference solution (estimated convergence order) vs.\ spatial mesh refinement
           in 2-D and 3-D.}
  \label{tab:PoissonDiracDirichlet}
  \vspace{0.5\baselineskip}
\setlength{\tabcolsep}{3pt}
\begin{tabular}{r rl rl}
\hline
& \multicolumn{2}{c}{2-D} & \multicolumn{2}{c}{3-D} \\
\hline
$r$ & $\|u-u_{ref}\|$ & ($p_{est}$) & $\|u-u_{ref}\|$ & ($p_{est}$) \\
\hline
        0 &    9.322030e-01 &         &   1.430697e-01 &         \\
        1 &    4.799837e-01 &(0.9577) &   1.004787e-01 &(0.5098) \\
        2 &    2.430508e-01 &(0.9817) &   7.117283e-02 &(0.4975) \\
        3 &    1.205341e-01 &(1.0118) &   5.086844e-02 &(0.4846) \\
        4 &    5.801228e-02 &(1.0550) &   3.618990e-02 &(0.4912) \\
        5 &    2.649574e-02 &(1.1306) &   2.373690e-02 &(0.6085) \\
        6 &    1.061522e-02 &(1.3196) &     ---        & ---     \\
\hline
\end{tabular}
\end{table}
\begin{figure} 
    \centering
    \begin{tabular}{cc}
         \includegraphics[width=0.45\linewidth]{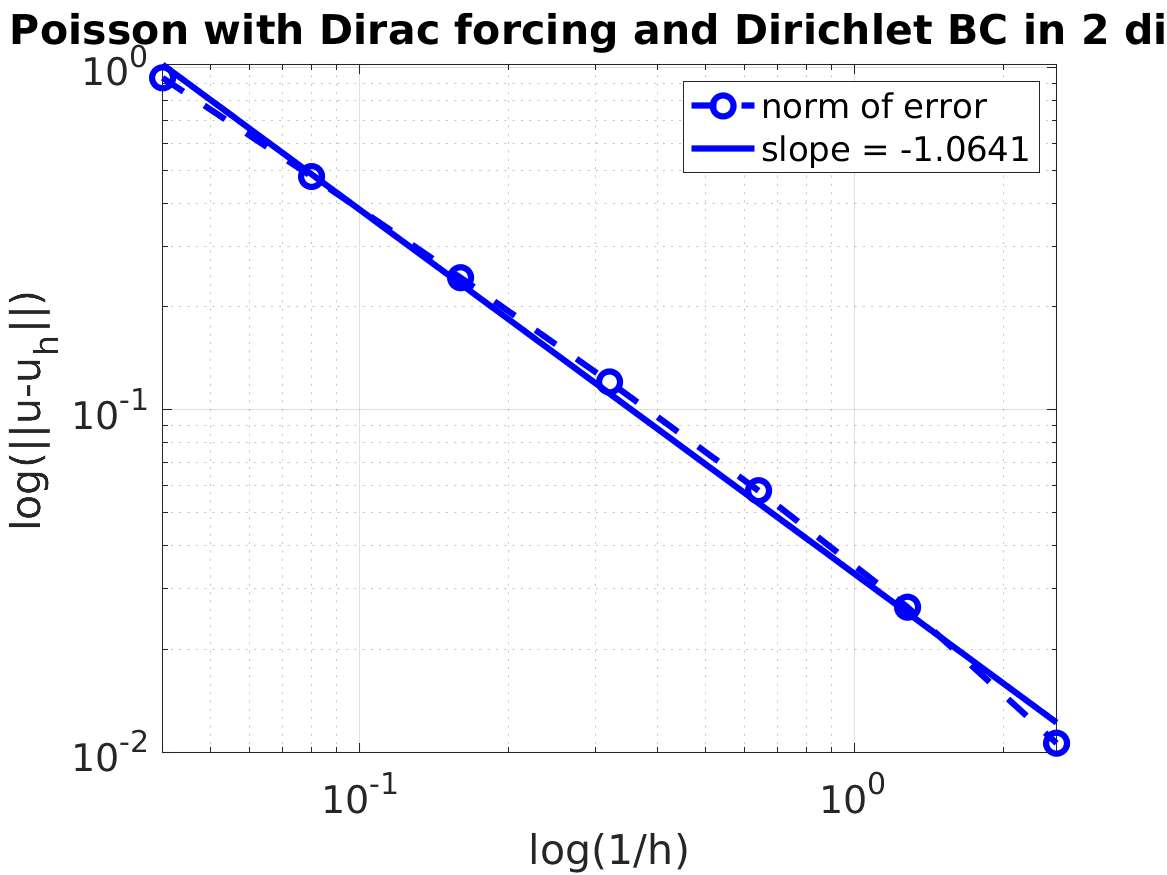} &
         \includegraphics[width=0.45\linewidth]{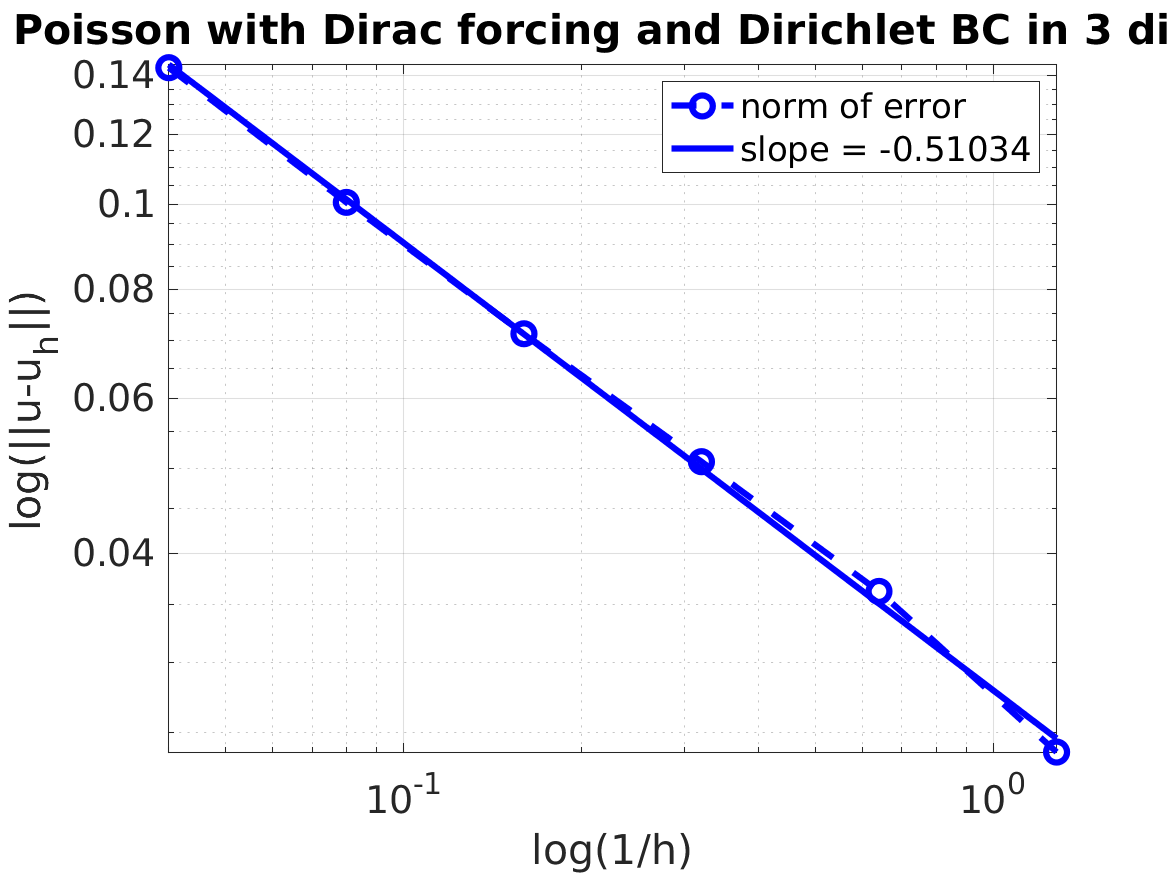} \\
         2-D &
         3-D                                                                              \\
    \end{tabular}
    \caption{Stationary diffusion-reaction equation with homogeneous Dirichlet boundary conditions and Dirac source term: observed $\log{|u-u_h|}$-error against $\log{1/h}$ where the slope estimates convergence order in 2-D and 3-D.}
    \label{fig:PoissonDiracDirichlet}
\end{figure}


\clearpage

\subsection{Time-dependent diffusion-reaction equation with homogeneous Dirichlet boundary conditions}
\label{subsec:transient}

The time-dependent measure-valued test problem is given by
the diffusion-reaction equation \eqref{eq:DREQ} for the concentration $u(\xx,t)$
\begin{align}
\label{equ:PDE}
    u_t - \nabla \cdot (D \, \nabla u) &= \kappa \, \delta_{\0}\!(\xx)
        &&\text{ for $\xx \in \Omega$ and $0 < t \leq t_{\text{final}}$,}              \\
    u(\xx,t) &= 0
        &&\text{ for $\xx \in \partial\Omega$ and $ 0 < t \leq t_{\text{final}} $,}               \\
    u(\xx,0) &= 0
        &&\text{ for $\xx \in \overline{\Omega}$ at $t = 0$,}
\end{align}
on the domain $\Omega = (-50,50)^d \subset \R^d$ in $d = 2, 3$ spatial dimensions.
The time-dependent test problem is inspired by the model of pollution
entering at an injection rate $\kappa$ as a concentrated point source
for 24~hours at the center $\0$ of the domain $\Omega$ and diffusing with diffusivity $D$.
The variables have the meaning and values collected in Table~\ref{tab:variables}.
Analogous to the analysis, the boundary condition is a homogeneous Dirichlet condition.
\begin{table}[!h] \centering
  \caption{Variables and their units for a pollution example.}
  \label{tab:variables}
  \vspace{0.5\baselineskip}
  \begin{tabular}{lll}
    \hline
          Variable & Definition & Value and/or Units \\
    \hline
          $d$        & spatial dimension               & $2$, $3$          \\
          $\xx \in \mathbb{R}^d$ & spatial position    &  km \\
          $t$        & time variable                   &  h                   \\
          $u(\xx,t)$ & concentration of pollutant      &  kg/km$^d$             \\
          $\Omega$   & spatial domain                  & $(-50,50)^d$           \\
          $L$        & domain length                   & 100 km \\
          $t_{\text{final}}$ & time domain             & 24 h \\
          $D$        & diffusivity coefficient         &  \rev{40 km$^2$/h}              \\
          $\kappa$   & injection rate \junk{in $s$}    &  10 kg km$^{-d}$ h$^{-1}$ \\
    \hline
  \end{tabular}
\end{table}

Now, we present convergence results of the finite volume method for the time-dependent diffusion-reaction equation with homogeneous Dirichlet boundary conditions and Dirac source term still conforming to the semi-discretization developed in Sections~\ref{sec:num_method}.
The implicit Euler with a small time step $\Delta t = 1/128$ and
the conjugate gradient method with a tight tolerance $10^{-12}$
are used to ensure that their errors do not obscure the spatial convergence behavior
in (\ref{equ:ResultHeatPersistent}) in Section~\ref{sec:HeatErrAnalPersistent}.

Table~\ref{tab:HeatDiracDirichlet} shows, for 2-D and 3-D,
the errors against the reference solution for several uniform refinement levels $r = 0, 1, 2, \ldots$
at three representative points in time $t = 4, 8, 12$ during the transient simulation.
In parentheses are listed the convergence order $p_{est}$ estimated as logarithm of the
ratio between the error norms of consecutive refinement levels,
estimated in the same way as described in the previous Subsection~\ref{subsec:stationary}.
As in the stationary results,
the results \eqref{equ:ResultHeatPersistentSemi} and
\eqref{equ:ResultHeatPersistent} in Section~\ref{sec:HeatErrAnalPersistent}
lead us to expect convergence orders of
$O(h |\log h|^{1/2}) \to O(h)$ as $h \to 0$ in 2-D and
$O(h^{1/2}) = O(h^{0.5})$ in 3-D.

Figure~\ref{fig:HeatDiracDirichlet} shows the log-log-plot of the convergence order estimate
and the linear fit of $\log{\|u-u_h\|}$ against $\log{1/h}$,
whose slope estimates the convergence order.

Both Table~\ref{tab:HeatDiracDirichlet} and Figure~\ref{fig:HeatDiracDirichlet}
show the numerical convergence order tending to $p = 1$ for 2-D,
while close to $p = 0.5$ in 3-D across all three times $t = 4, 8, 12$.
The values agree with the expected values in the results of Section~\ref{sec:HeatErrAnalPersistent}
for 2-D and 3-D.
\rev{As in the stationary case, the increase of $p_{est}$ at the finest level $r=5$ in 3-D
is due to the self-comparison bias discussed in Subsection~\ref{subsec:stationary}.}

\junk{%
In the time-dependent diffusion-reaction equation of \ref{equ:PDE}, we enable a smooth test problem by setting $s \equiv 0$,  
and, stated here for concreteness in 2-D,
\begin{alignat}{3}
  &f(x,y,t)&= (2t/\tau^2)e^{-(t/\tau)^2} &\cos^2(\pi x/L)\cos^2(\pi y/L)
              \label{equ:test} \\
  &        &+ \, (1-e^{-(t/\tau)^2}) \,
  ( &(-2\pi^2/L^2)\cos(2\pi x/ L)\cos^2(\pi y /L) + \mbox{} \nonumber \\
  &&&(-2\pi^2/L^2)\cos^2(\pi x/L)\cos(2\pi y/ L) ),         \nonumber
\end{alignat}
on right-hand side of the time-dependent test problem 
so that the problem admits the known analytical solution
\begin{equation} \label{equ:utrue}
  u(x,y,t) = (1-e^{-(t/\tau)^2}) \cos^2(\pi x/L) \cos^2(\pi y/L),
\end{equation}
where $\tau$ controls the time scale of the approach to steady state and
the length of the spatial domain in each dimension is $L$.
}

\junk{%
The measure-valued test problem is set by 
modeling the injection of pollution at the center $0$ of the domain $\Omega$
a point source by the measure-valued Dirac delta distribution as
\begin{equation}
  \varphi(\xx,t) = \kappa \, \delta_{0}\!(\xx)
\end{equation}
with injection rate $\kappa$.
}

\junk{
\clearpage

Table~\ref{tab:HeatDiracDirichlet}
contains results for the time-dependent diffusion-reaction 
equation with Dirichlet boundary conditions and smooth source term
using the finite volume method (FVM) with spatial discretization described.
Errors are computed at $t= 4, 8, 12$ as the solution evolves in time. We observe that the FVM method has a numerical convergence order near $p=2$ for this smooth test problems in 2-D and 3-D for all times the errors are computed. We note that in 1-D when $t=8,12$ we see the order diminished.

\begin{table} \centering
  \caption{Time-dependent diffusion-reaction equation with Dirichlet boundary conditions and smooth source term:
           observed $L^2$-error against reference solution (estimated convergence order) vs.\ spatial mesh refinement
           at three points in time.}
  \label{tab:HeatSmoothDirichlet}
  \vspace{0.5\baselineskip}
\begin{tabular}{r rl rl rl}
\hline
& \multicolumn{2}{c}{$t=4$} & \multicolumn{2}{c}{$t=8$} & \multicolumn{2}{c}{$t=12$} \\
$r$ & $\|u-u_{ref}\|$ & ($p_{est}$) & $\|u-u_{ref}\|$ & ($p_{est}$) & $\|u-u_{ref}\|$ & ($p_{est}$) \\
\hline
& \multicolumn{6}{l}{diffusion with smooth forcing and Dirichlet BC in 1 dim} \\
 0 & 1.491461e-01&         & 3.832955e-01&         & 4.897754e-01&         \\
 1 & 3.781538e-02&(1.9797) & 9.552107e-02&(2.0046) & 1.198929e-01&(2.0304) \\
 2 & 9.564097e-03&(1.9833) & 2.396423e-02&(1.9949) & 3.046892e-02&(1.9763) \\
 3 & 2.428481e-03&(1.9776) & 5.954254e-03&(2.0089) & 7.742288e-03&(1.9765) \\
 4 & 6.232652e-04&(1.9621) & 1.399095e-03&(2.0894) & 1.919904e-03&(2.0117) \\
 5 & 1.554308e-04&(2.0036) & 3.214421e-04&(2.1219) & 4.638603e-04&(2.0493) \\
 6 & 3.441523e-05&(2.1752) & 2.144059e-04&(0.5842) & 1.367695e-04&(1.7619) \\
 7 & 8.944486e-06&(1.9440) & 1.462859e-04&(0.5516) & 6.227814e-05&(1.1349) \\
\hline
& \multicolumn{6}{l}{diffusion with smooth forcing and Dirichlet BC in 2 dim} \\
 0 & 1.515645e+00&         & 4.051172e+00&         & 5.436974e+00&         \\
 1 & 4.045844e-01&(1.9054) & 1.060356e+00&(1.9338) & 1.395441e+00&(1.9621) \\
 2 & 1.031694e-01&(1.9714) & 2.686610e-01&(1.9807) & 3.529600e-01&(1.9831) \\
 3 & 2.605593e-02&(1.9853) & 6.770164e-02&(1.9885) & 8.888982e-02&(1.9894) \\
 4 & 6.591243e-03&(1.9830) & 1.705034e-02&(1.9894) & 2.225350e-02&(1.9980) \\
 5 & 1.659132e-03&(1.9901) & 3.949232e-03&(2.1102) & 5.702087e-03&(1.9645) \\
 6 & 4.069750e-04&(2.0274) & 1.053727e-03&(1.9061) & 1.528594e-03&(1.8993) \\
\hline
& \multicolumn{6}{l}{diffusion with smooth forcing and Dirichlet BC in 3 dim} \\
 0 & 1.250052e+01&         & 3.396712e+01&         & 4.634873e+01&         \\
 1 & 3.483291e+00&(1.8435) & 9.269360e+00&(1.8736) & 1.239357e+01&(1.9029) \\
 2 & 8.970943e-01&(1.9571) & 2.377737e+00&(1.9629) & 3.164954e+00&(1.9693) \\
 3 & 2.269976e-01&(1.9826) & 6.033503e-01&(1.9785) & 8.011411e-01&(1.9821) \\
 4 & 5.743176e-02&(1.9828) & 1.515559e-01&(1.9931) & 2.050265e-01&(1.9662) \\
 5 & 1.432804e-02&(2.0030) & 3.875003e-02&(1.9676) & 5.372061e-02&(1.9323) \\
\hline
\end{tabular}
\end{table}
\begin{figure}
    \centering
        \begin{tabular}{ccc}
         \includegraphics[width=0.3\linewidth]{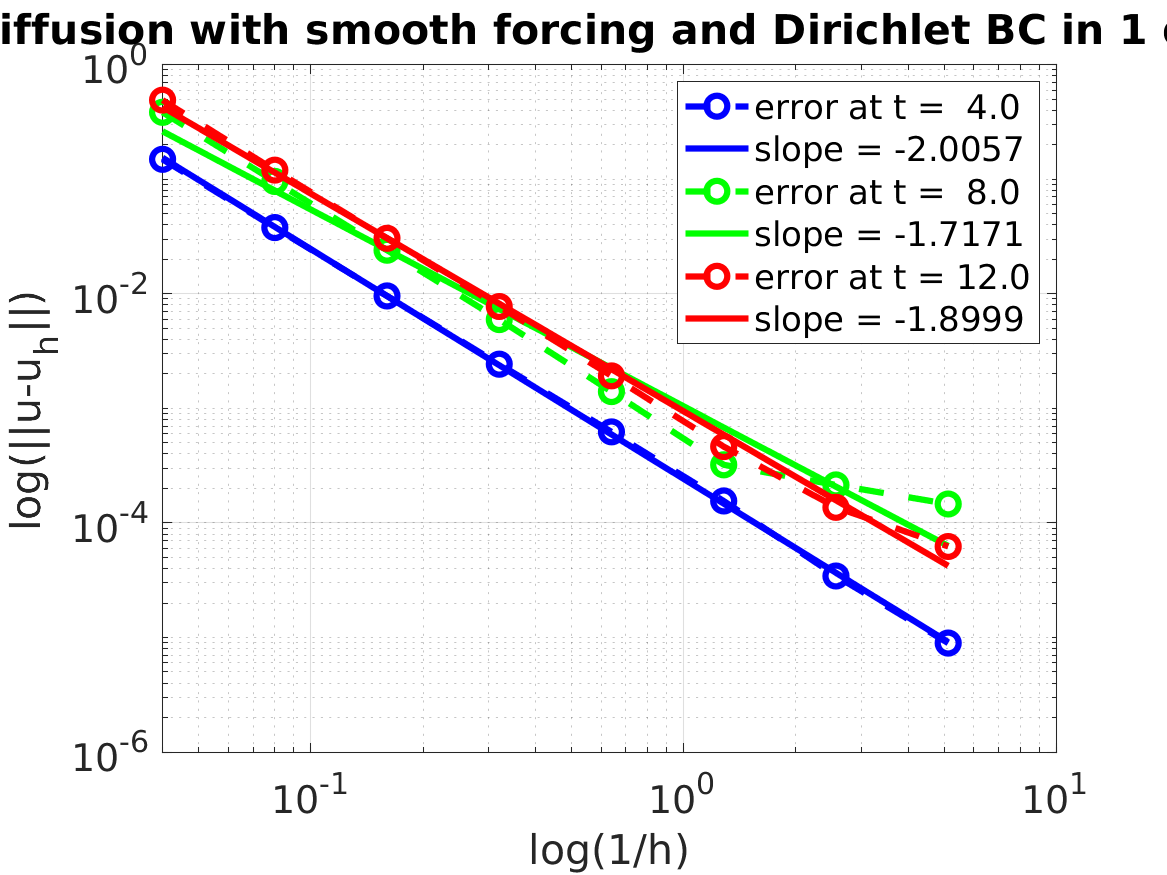} &
         \includegraphics[width=0.3\linewidth]{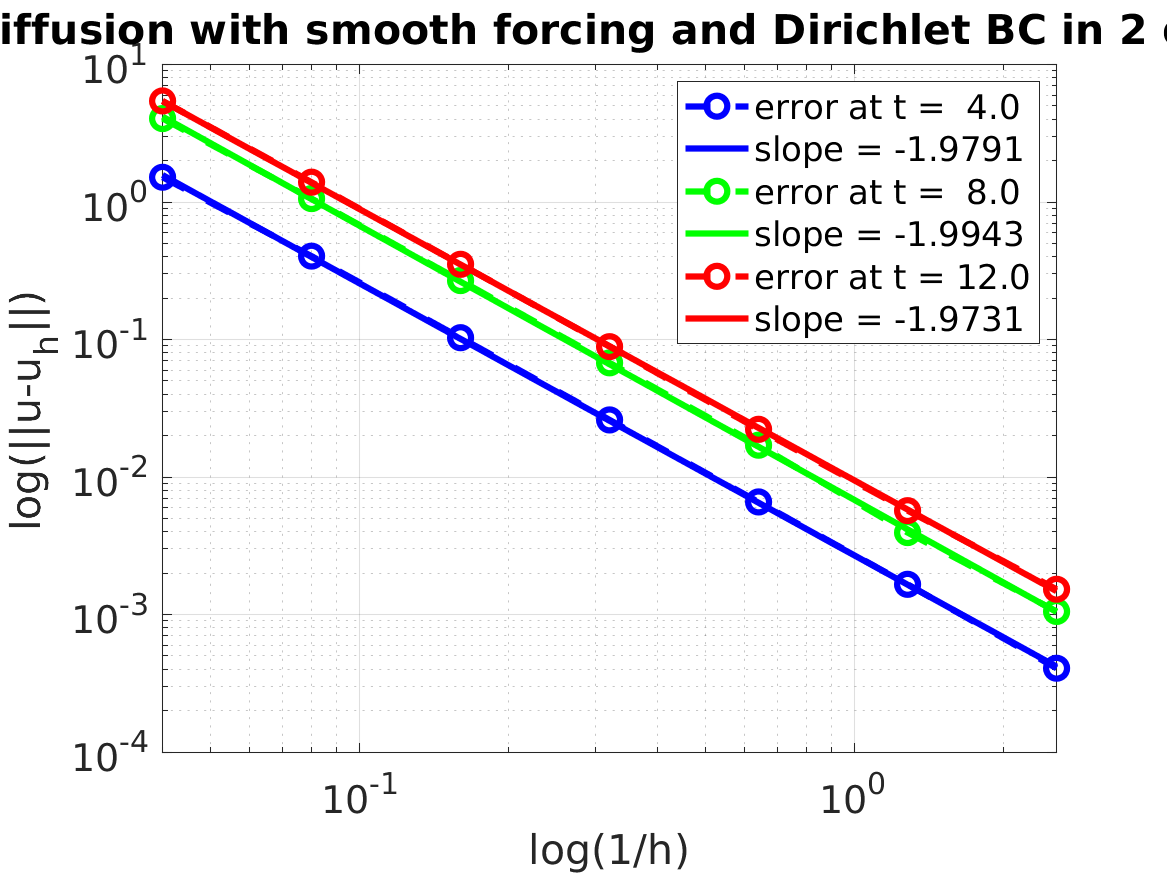} &
         \includegraphics[width=0.3\linewidth]{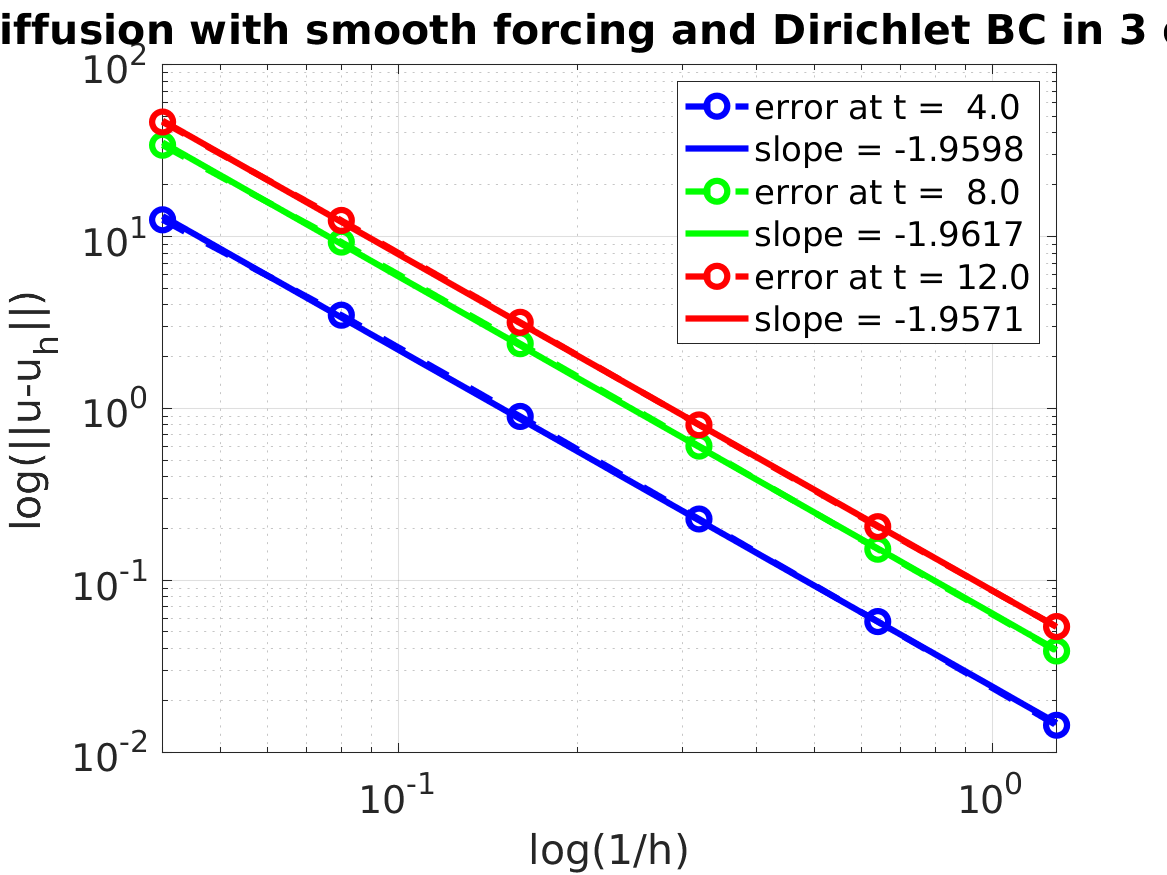} \\
         1-D &
         2-D &
         3-D
        \end{tabular}  
    \caption{Time-dependent diffusion-reaction equation with Dirichlet boundary conditions and smooth source term: observed $\log{|u-u_h|}$-error against $\log{1/h}$ where the slope estimates convergence order in 1-D, 2-D, 3-D.}
    \label{fig:placeholder}
\end{figure}
}

\begin{table} \centering
  \caption{Time-dependent diffusion-reaction equation with Dirichlet boundary conditions and one point source:
           observed $L^2$-error against reference solution (observed convergence order) vs.\ spatial mesh refinement
           at three points in time.}
  \label{tab:HeatDiracDirichlet}
  \vspace{0.5\baselineskip}
\begin{tabular}{r rl rl rl}
\hline
& \multicolumn{2}{c}{$t=4$} & \multicolumn{2}{c}{$t=8$} & \multicolumn{2}{c}{$t=12$} \\
$r$ & $\|u-u_{ref}\|$ & ($p_{est}$) & $\|u-u_{ref}\|$ & ($p_{est}$) & $\|u-u_{ref}\|$ & ($p_{est}$) \\
\hline
& \multicolumn{6}{l}{diffusion with Dirac forcing and Dirichlet BC in 2 dim} \\
 0 & 4.429665e-01&         & 4.394271e-01&         & 4.467623e-01&         \\
 1 & 2.150050e-01&(1.0428) & 2.160408e-01&(1.0243) & 2.235983e-01&(0.9986) \\
 2 & 1.072135e-01&(1.0039) & 1.082159e-01&(0.9974) & 1.131661e-01&(0.9825) \\
 3 & 5.422757e-02&(0.9834) & 5.479832e-02&(0.9817) & 5.756817e-02&(0.9751) \\
 4 & 2.801808e-02&(0.9527) & 2.828863e-02&(0.9539) & 2.959463e-02&(0.9599) \\
 5 & 1.491316e-02&(0.9098) & 1.501291e-02&(0.9140) & 1.549514e-02&(0.9335) \\
 6 & 7.685259e-03&(0.9564) & 7.707547e-03&(0.9619) & 7.816026e-03&(0.9873) \\
\hline
& \multicolumn{6}{l}{diffusion with Dirac forcing and Dirichlet BC in 3 dim} \\
 0 & 1.384470e-01&         & 1.385138e-01&         & 1.386870e-01&         \\
 1 & 9.838274e-02&(0.4929) & 9.840875e-02&(0.4932) & 9.854894e-02&(0.4929) \\
 2 & 7.037154e-02&(0.4834) & 7.038765e-02&(0.4835) & 7.044464e-02&(0.4844) \\
 3 & 5.060385e-02&(0.4757) & 5.060914e-02&(0.4759) & 5.062910e-02&(0.4765) \\
 4 & 3.611849e-02&(0.4865) & 3.611993e-02&(0.4866) & 3.612553e-02&(0.4869) \\
 5 & 2.372452e-02&(0.6064) & 2.372477e-02&(0.6064) & 2.372576e-02&(0.6066) \\
\hline
\end{tabular}
\end{table}
\begin{figure}\label{fig:DiffDiracloglog}
    \centering
        \begin{tabular}{ccc}
         \includegraphics[width=0.45\linewidth]{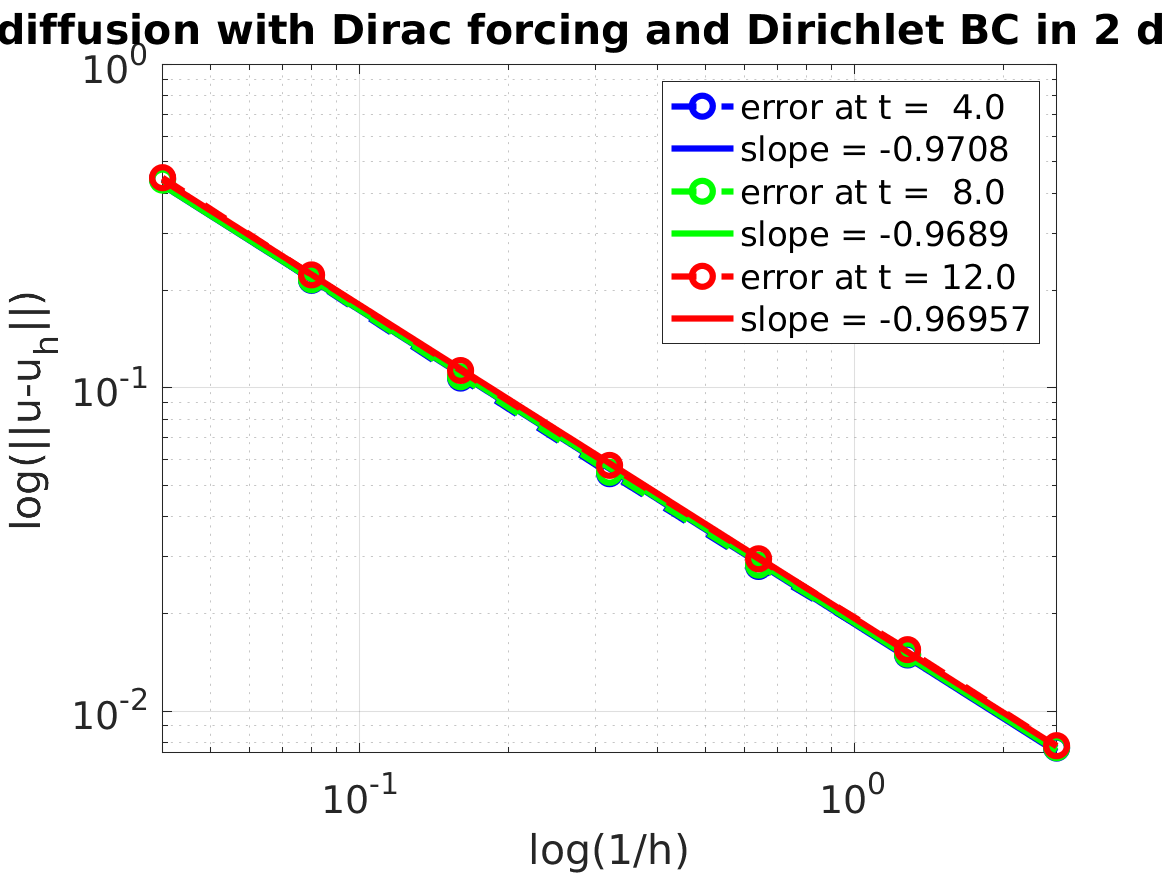} &
         \includegraphics[width=0.45\linewidth]{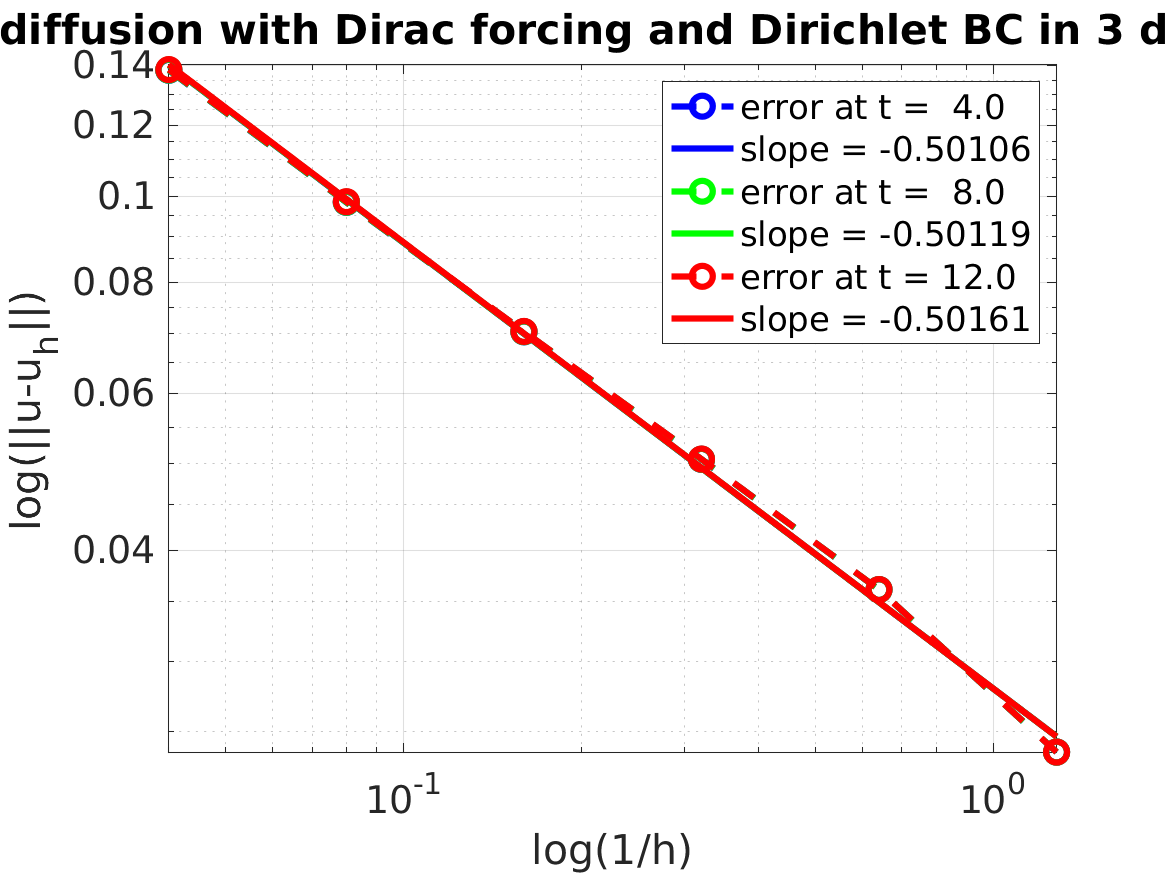} \\
         2-D &
         3-D
        \end{tabular}  
    \caption{Time-dependent diffusion-reaction equation with Dirichlet boundary conditions and one point source: observed $\log{|u-u_h|}$-error against $\log{1/h}$ where the slope estimates convergence order in 2-D and 3-D.}
   \label{fig:HeatDiracDirichlet}
\end{figure}



\section*{Acknowledgments}
We acknowledge the computational resources in the
UMBC High Performance Computing Facility (\url{hpcf.umbc.edu})
and the financial contributions from NIH, NSF, CIRC, and UMBC for this work.
MK is indebted to Petr Knobloch for fruitful discussions. Moreover, he acknowledges  support of the  CSF project 24--10366S. 

\bibliography{MeasureFVM_EECT2026,new_references}

\end{document}